%% file: main.tex
\documentclass[a4paper, english]{amsart}

\usepackage[utf8]{inputenc}
\usepackage[T1]{fontenc}
\usepackage[final]{microtype}
\usepackage{lmodern}
\usepackage{babel}

\usepackage[footskip=25pt, hcentering]{geometry}
\usepackage{fancyhdr}
\usepackage[hidelinks, pdfusetitle, breaklinks]{hyperref}
\usepackage{graphicx}
\usepackage{xcolor}
\usepackage{enumitem}
\usepackage{xparse}
\usepackage[style=american]{csquotes}
\usepackage[%
    backend=biber, 
    giveninits=true, 
    maxbibnames=4, 
    maxcitenames=4,
    isbn=false,
    date=year,
    doi=false
]{biblatex}
\renewbibmacro{in:}{}
\renewrobustcmd*{\bibinitdelim}{\,}
\DeclareFieldFormat[article,inbook,incollection,inproceedings,patent,thesis,unpublished]{title}{#1\isdot}
\DeclareFieldFormat[article]{volume}{\textbf{#1}}
\DeclareFieldFormat[article]{number}{no.\ #1}
\renewrobustcmd*{\newunitpunct}{\addcomma\space}
\renewbibmacro{volume+number+eid}{%
    \printfield{volume}%
    \setunit{\addcomma\space}%
    \printfield{number}%
    \setunit{\addcomma\space}%
    \printfield{eid}}

\usepackage[centertags]{amsmath}
\usepackage{nccmath}
\usepackage{array}
\usepackage{quiver}
\usepackage{mathtools}

\usepackage{amssymb}
\usepackage{amsfonts}
\usepackage[mathscr]{eucal}
\usepackage{upgreek}
\usepackage{dsfont}
\usepackage{bbold}

\usepackage{amsthm}
\usepackage{thmtools}
\declaretheorem[style=definition, numberwithin=section]{definition}
\declaretheorem[style=definition, sibling=definition]{example}

\declaretheorem[style=plain, sibling=definition]{proposition}
\declaretheorem[style=plain, sibling=definition]{theorem}
\declaretheorem[style=plain, numbered=no, name=Theorem]{ntheorem}
\declaretheorem[style=plain, sibling=definition]{lemma}
\declaretheorem[style=plain, sibling=definition]{corollary}

\numberwithin{equation}{section}

\input{commands}

\title{Sketchable Infinity Categories}
\author{Carles Casacuberta \and Javier J.\ Gutiérrez \and David Martínez-Carpena}

\begin{document}

\begin{abstract}
    A sketch is a category equipped with specified collections of cones and cocones.
    Its models are functors to the category of sets that send the distinguished cones and cocones to limit cones and colimit cocones, respectively.
    Sketches provide a categorical formalization of theories, interpreting logical operations in terms of limits and colimits.
    Gabriel and Ulmer showed that categories of models of sketches involving only cones (called limit sketches) are precisely the locally presentable categories, while Lair extended this correspondence to sketches including both cones and cocones, thereby characterizing accessible categories.

    In this article, we discuss a homotopy-coherent generalization of sketches in the context of $\infty$-categories and prove that presentable $\infty$-categories are the $\infty$-categories of models of limit sketches, whereas accessible $\infty$-categories arise as the $\infty$-categories of models of arbitrary sketches.
    As illustrations, we make the corresponding sketches explicit for a wide range of $\infty$-categories, including complete Segal spaces, $\infty$-operads, $A_\infty$-algebras, $E_\infty$-algebras, spectra, and higher sheaves.
\end{abstract}

\maketitle

\setcounter{tocdepth}{1}
\tableofcontents

\section{Introduction}

The notion of locally presentable categories, introduced by Gabriel and Ulmer in~\autocite{Gabriel1971}, arose as an abstraction of the idea of presenting mathematical structures by  generators and relations.
A~locally presentable category is a cocomplete category generated under filtered colimits by a set of compact objects.
By relaxing the assumption of cocompleteness and only requiring the existence of filtered colimits, one obtains the broader notion of \emph{accessible} categories.
For example, the category of fields is accessible but not locally presentable \autocite[Example~2.3(5)]{Adamek1994}.

A~\emph{sketch}, in the sense of \textcite{Bastiani1972}, consists of a small category equipped with a specified collection of cones and cocones.
A~\emph{model} of a sketch $\Sigma$ is a functor from $\Sigma$ to the category of sets that sends each distinguished cone to a limit cone and each distinguished cocone to a colimit cocone.
If the collection of distinguished cocones is empty, then $\Sigma$ is called a \emph{limit} sketch.
Likewise, if the collection of distinguished cones is empty, $\Sigma$ is a \emph{colimit} sketch.

There are numerous examples in the literature of categories modeled by limit sketches, including colored operads and the models of any Lawvere theory.
Through categorical logic, one can show that the models of limit sketches are equivalent to the models of essentially algebraic theories.
A~foundational result of Gabriel and Ulmer~\autocite{Gabriel1971}, revisited by Adámek and Rosický in~\autocite{Adamek1994}, highlights the importance of sketches by showing that locally presentable categories are precisely the categories of models of limit~sketches.
Extending this correspondence, Lair~\autocite{Lair1981} proved that accessible categories are exactly those modeled by mixed sketches, that is, sketches equipped with both cones and cocones; see also~\autocite[Theorem~2.58]{Adamek1994}.
Moreover, Makkai and Paré~\autocite[Subsection~5.1.3]{Makkai1989} exhibited a 2-adjunction between the small 2-categories of categories and of sketches.

The theory of limit (and colimit) sketches was first considered in homotopical settings through the framework of Quillen model categories~\autocite{Badzioch2002, Caviglia2016, Rosicky2007, Rosicky2013}.
In~\autocite{Badzioch2002}, Badzioch introduced the notion of homotopy models of algebraic theories on spaces, and Rosický~\autocite{Rosicky2007} extended this work by considering homotopy models of simplicial algebraic theories on spaces.
Moreover, he established a correspondence between homotopy models of simplicial limit sketches and homotopy locally finitely presentable simplicial categories.
In~\autocite{Rosicky2013}, Rosický further studied models of finite weighted enriched limit sketches in combinatorial monoidal model categories, proving that their homotopy models are again combinatorial, and hence locally presentable, under mild assumptions.
Subsequent developments include Corrigan-Salter’s generalization of Badzioch’s results to the multi-sorted setting~\autocite{CorriganSalter2015}, the work of Caviglia and Horel on homotopy models of limit sketches whose cones have finite connected diagrams~\autocite{Caviglia2016}, and Marelli’s construction of a homotopy limit 2-sketch modeling derivators~\autocite{Marelli2021}.

Joyal~\autocite{JoyalChicago,JoyalCRM} developed the theory of $\infty$-categories modeled as quasi-categories, which was followed by the influential work of Lurie~\autocite{Lurie2009}.
In his notes, Joyal introduced the notion of limit sketches on quasi-categories and outlined several examples of higher categories modelled by them.
Since then, very few authors have pursued the theory of limit sketches in this setting.
In~\autocite{Chu2021}, Chu and Haugseng studied algebraic patterns on $\infty$-categories, a particular kind of limit sketch endowed with additional structure.
This extra structure enables a direct connection between algebraic patterns and a certain class of higher monads.
In unpublished work~\autocite{Macpherson2021}, Macpherson examined colimit sketches with models in $\infty$-categories, introducing a refined notion of colimit sketch with constructions, which additionally specifies the colimits required to exist in the target category of models.

Within the framework of $\infty$-categories, Joyal~\autocite{JoyalChicago,JoyalCRM} and Lurie~\autocite{Lurie2009} pioneered the study of presentability and accessibility.
Both defined an accessible quasi-category as one equivalent to an $\Ind_\kappa$-category~(see~\autocite[Definition 5.3.5.1]{Lurie2009} for details) of some small $\infty$-category and some regular cardinal $\kappa$.
However, their definitions of presentable quasi-categories differ.
Joyal defined presentable $\infty$-categories as those equivalent to the categories of models of limit sketches, whereas Lurie defined them as those that are both cocomplete and accessible.
In his notes~\autocite{JoyalChicago,JoyalCRM}, Joyal asserts that his definition is equivalent to the condition of being cocomplete and accessible.
Modern literature generally follows Lurie’s conventions, and, to the best of our knowledge, no proof of Joyal’s claim has yet appeared in print.
Nevertheless, there are several partial results in this direction: Lurie~\autocite[Proposition 5.5.8.10]{Lurie2009} proved that the $\infty$-category of models of a higher Lawvere theory (that is, a sketch with only product cones) is presentable, and Chu and Haugseng~\autocite[Lemma 2.11]{Chu2021} showed that the $\infty$-category of models of an algebraic pattern is presentable.
Macpherson established in~\autocite[Proposition 3.4.1]{Macpherson2021} that the $\infty$-category of models of any colimit sketch is presentable.

The goal of this article is to extend to the setting of higher category theory both the equivalence between sketchability and accessibility and the equivalence between limit sketchability and presentability.
Our main results can be summarized as follows; see Corollaries~\ref{cor:main01} and~\ref{cor:main02}.

\begin{ntheorem}
    Let $\cC$ be an $\infty$-category.
    Then
    \begin{enumerate}[label={\rm (\roman*)}]
        \item $\cC$ is presentable if and only if $\cC$ is equivalent to the category of models of a limit sketch.
        \item $\cC$ is accessible if and only if $\cC$ is equivalent to the category of models of a sketch.
    \end{enumerate}
\end{ntheorem}

In Section~\ref{sec:sketches} we provide explicit limit sketches for the following presentable $\infty$-categories: spectra, Segal spaces and complete Segal spaces, $A_{\infty}$-spaces and $A_{\infty}$-rings, $E_{\infty}$-spaces and $E_{\infty}$-rings, infinite loop spaces, dendroidal Segal spaces and complete dendroidal Segal spaces, and higher sheaves.
We also give an example of a mixed sketch whose models are the nonempty path-connected spaces whose fundamental group is perfect, which form an accessible $\infty$-category that is not presentable.

Our proofs rely on a higher-categorical analog of the following classical characterization for ordinary categories.
Given a regular cardinal $\kappa$, a functor is called \emph{$\kappa$-flat} if its left Kan extension along the Yoneda embedding preserves $\kappa$-small limits.
Then the following holds:
\begin{equation}%
    \label{eq:flat}
    \text{\emph{A functor is $\kappa$-flat if and only if it is a $\kappa$-filtered colimit of representable functors}.}
\end{equation}

This characterization was first established by Kelly~\autocite{Kelly1982} and subsequently employed by Makkai and Paré~\autocite{Makkai1989} in their proof of the equivalence between sketchability and accessibility for ordinary categories.
The same result was later included by Borceux in~\autocite{Borceux1994} and by Adámek and Rosický in~\autocite{Adamek1994}, and was further generalized by Adámek, Borceux, Lack, and Rosický~\autocite{Adamek2002}, by replacing regular cardinals with classes of categories.
More recently, Lack and Tendas~\autocite{Lack2022a} established an enriched version of the same claim~\eqref{eq:flat}.

Our principal technical contribution in this article is the following characterization of flatness, where we write $\hRelSlice{\cA}{F}$ for the relative slice of a presheaf $F \colon \cA \to \Spaces$ along the Yoneda embedding, and $\Lan F$ for the corresponding left Kan extension; see Theorem~\ref{thm:flat<->filtcolim}.

\begin{ntheorem}%
    Let $\cA$ be a small $\infty$-category, $\kappa$ a regular cardinal, and $F \colon \Op{\cA} \to \Spaces$ a presheaf.
    The following statements are equivalent:
    \begin{enumerate}[label={\rm (\roman*)}]
        \item $F$ is $\kappa$-flat.
        \item $\Lan F$ preserves $\kappa$-small limits of representables.
        \item $\hRelSlice{\cA}{F}$ is a $\kappa$-filtered $\infty$-category.
        \item $F$ is a $\kappa$-filtered colimit of representables.
    \end{enumerate}
\end{ntheorem}

Further work on flatness in higher categories has also appeared in the literature.
Raptis and Sch\"appi~\autocite{Raptis2022} proved a version of~\eqref{eq:flat} for presheaves valued in arbitrary $\infty$-topoi, restricted to the finite case $\kappa = \aleph_0$.
An unpublished preprint by Rezk~\autocite{Rezk2021} contains a generalization of~\eqref{eq:flat} to classes of categories.

\medskip
\noindent
\textbf{Organization of the paper.}
The structure of the paper is as follows.
Section~\ref{sec:preliminaries} reviews the necessary background on $\infty$-categories, including notation, size conventions, and the main results concerning accessibility and presentability.
Section~\ref{sec:sketches} introduces the notion of higher sketches, generalizing the classical concept to the $\infty$-categorical setting, and provides several illustrative examples.
In Section~\ref{sec:flat-functors}, we study flat functors and prove the higher-categorical analog of the classical characterization of $\kappa$-flatness (Theorem~\ref{thm:flat<->filtcolim}).
As consequences, we obtain new characterizations of accessible and presentable $\infty$-categories (Corollary~\ref{lem:acc-flat} and Corollary~\ref{cor:presentable<->Cont}).
Section~\ref{sec:limit-sketchable} establishes the equivalence between limit-sketchability and presentability for $\infty$-categories, culminating in Corollary~\ref{cor:main01}.
We study limit sketches with models in arbitrary $\infty$-categories using a construction which, at the level of models, corresponds to Lurie’s tensor product of presentable $\infty$-categories.
Section~\ref{sec:sketchable} extends the argument to arbitrary sketches, proving the equivalence between sketchability and accessibility (Corollary~\ref{cor:main02}).

\medskip
\noindent
\textbf{Acknowledgements.}
We are indebted to George Raptis for very useful comments and suggestions, and we also thank Ji\v{r}í Rosický and Ivan di Liberti for valuable conversations.
The authors acknowledge support from the Departament de Recerca i Universitats de la Generalitat de Catalunya (2021 SGR 00697) and from the Agencia Estatal de Investigación (MCIN/AEI) under grants PID2020-117971GB-C22, PID2024-155646NB-I00, and Europa Excelencia grant EUR2023-143450.

\section{Preliminaries}%
\label{sec:preliminaries}

\subsection{Quasi-categories}

In this work, we implicitly use the formalism of quasi-categories~\autocite{Cisinski2019,Joyal2002,Lurie2009} for $\infty$\nobreakdash-category theory.
Thus, we use the term \textit{$\infty$-groupoid} to refer to a Kan complex.
Every $\infty$-category $\cC$ has a collection of objects $\Obj (\cC)$, but we denote the fact that $x$ is an object of $\cC$ by writing $x \in \cC$.
We denote $\sDelta{n} = \Deltabf(-, [n])$, where $\Deltabf$ is the simplex category and $[n] = \{0 < 1 < \cdots < n\}$.
If $\cC$ and $\cD$ are $\infty$-categories, then the simplicial set $\Fun{\cC}{\cD}$, whose $n$-simplices are the maps $\cC \times \Delta^n \to \cD$, is an $\infty$-category.
We call $\Fun{\cC}{\cD}$ the $\infty$\nobreakdash-category of functors from $\cC$ to~$\cD$, and denote an object $F \in {\Fun{\cC}{\cD}}$ by $F \colon \cC \to \cD$.
A~natural transformation between two functors $F, G \colon \cC \to \cD$ is a $1$-simplex $\alpha \colon \cC \times \sDelta{1} \to \cD$ of $\Fun{\cC}{\cD}$ whose restriction to $\cC\times\{0\}$ is equal to $F$ and whose restriction to $\cC\times\{1\}$ is equal to~$G$.

Since the nerve of any category is a quasi-category, we treat categories as quasi-categories without specifying the nerve functor in the notation, if no confusion can arise.
In particular, $\Delta^n$ can be viewed as an $\infty$-category, since $\Delta^n$ is the nerve of the poset $0\to\cdots\to n$.

Given two objects $x, y$ of an $\infty$-category~$\cC$, we denote by $\Map{\cC} (x, y)$ the $\infty$-groupoid of morphisms (or \textit{mapping space}) from $x$ to~$y$, which is defined by the following pullback of simplicial sets:
\[
    \begin{tikzcd}[ampersand replacement=\&,cramped]
        {\Map{\cC} (x, y)} \& {\Fun{\sDelta{1}}{\cC}} \\
        {\sDelta{0}} \& {\cC\times\cC,}
        \arrow[from=1-1, to=1-2]
        \arrow[from=1-1, to=2-1]
        \arrow["{(\Ev_0,\, \Ev_1)}", from=1-2, to=2-2]
        \arrow[""{name=0, anchor=center, inner sep=0}, "{(x, y)}", hook, from=2-1, to=2-2]
        \arrow["\lrcorner"{anchor=center, pos=0.125}, draw=none, from=1-1, to=0]
    \end{tikzcd}
\]
where $(\Ev_0, \Ev_1)$ is obtained by applying $\Fun{-}{\cC}$ to the map $(\sFace{1}{0} , \sFace{1}{1}) \colon \sDelta{1} \to \sDelta{0} \times \sDelta{0}$.
We denote by $f \colon x \to y$ the fact that $f \in {\Map{\cC} (x, y)}$.
A mapping space of a functor $\infty$-category $\Map{\sFun{\cC}{\cD}} (F, G)$ has as 0-simplices the natural transformations from $F$ to~$G$.
We denote these by $\alpha \colon F \Rightarrow G$.

Every $\infty$-category $\cC$ has a \textit{homotopy category} $\ho (\cC)$, with the same objects as $\cC$ and with $\ho(\cC)(x,y) = \pi_0 \Map{\cC} (x, y)$.
Given a morphism $f \colon x \to y$, we denote by $[f]$ the corresponding morphism in~$\ho (\cC)$.
A morphism $f \colon x \to y$ in $\cC$ is an \textit{isomorphism} if it is invertible in~$\ho (\cC)$.

We call an object \textit{unique} with a property when it is unique up to isomorphism among those sharing the property.
In the case of a morphism $f \colon x \to y$, we call it \textit{unique} with a property if the subspace of $\Map{\cC} (x, y)$ of those morphisms sharing the given property is contractible.
For example, the inverse of an isomorphism is unique.

With this convention, given two composable morphisms $f \colon x \to y$ and $g \colon y \to z$ in $\cC$, there is a unique morphism $h \colon x \to z$ such that $[h] = [g] \circ [f]$.
Composition can also be studied at the level of mapping spaces, where there is a unique (up to natural isomorphism) composition functor
\[
    - \circ - \colon
    \Map{\cC} (y, z) \times \Map{\cC} (x, y)
    \longrightarrow \Map{\cC} (x, z),
\]
defined by the construction given in~\autocite[\S\,45.6]{Rezk2022}.
It has the expected properties; namely, it is associative up to homotopy, and it matches with the composition defined in~$\ho(\cC)$.

\subsection{Cardinality assumptions}

The main concepts studied in this article are related to sizes of $\infty$\nobreakdash-categories.
Regular cardinals are assumed to be infinite.
For an uncountable regular cardinal~$\kappa$, an $\infty$-groupoid $\cX$ is called \textit{$\kappa$-small} if $\pi_0(\cX)$ and the homotopy groups $\pi_n (\cX, x)$ have  cardinality smaller than $\kappa$ for each $x \in \cX$ and each $n \geq 1$.
If $\kappa = \aleph_0$, then an $\infty$-groupoid $\cX$ is called \textit{$\aleph_0$-small} if it is a homotopy retract of a finite simplicial set.
An $\infty$-category is called \textit{locally $\kappa$-small} if all its mapping spaces are $\kappa$-small $\infty$-groupoids.
Furthermore, an $\infty$-category is called \textit{$\kappa$-small} if it is locally $\kappa$-small and its set of isomorphism classes of objects has cardinality smaller than~$\kappa$.
This definition is found with the name of \emph{essentially $\kappa$-small} $\infty$-category in some references such as~\autocite{Lurie2009}, but we follow the conventions of~\autocite{Anel2022, Cisinski2019}.

We assume the existence of a strongly inaccessible cardinal~$\kappa$, and call \emph{small sets} (or sometimes just \emph{sets}) the sets with cardinality smaller than $\kappa$.
An $\infty$-category will be called \textit{small} (resp.\ \textit{locally small}) if it is $\kappa$-small (resp.\ locally $\kappa$-small).
The locally small $\infty$-category of all small $\infty$-groupoids is denoted by $\Spaces$, and the one of all small $\infty$-categories is denoted by~$\iCat$.

As is common in the literature, the isomorphisms between $\infty$-categories or $\infty$-groupoids, viewed as objects of $\iCat$, will be called \textit{equivalences}.

If $\cK$ is small and $\cC$ is locally small, then $\Fun{\cK}{\cC}$ is a locally small $\infty$-category~\autocite[Example~5.4.1.8]{Lurie2009}.
Throughout this paper, unless explicitly specified, all $\infty$-categories are assumed to be locally small.
In the case where we need some $\infty$-category which is not necessarily locally small, it will be called a \textit{large} $\infty$-category.

\subsection{Notation and basic constructions}

A functor $F  \colon \cC \to \cD$ is \textit{essentially surjective} if for every object $y \in \cD$ there exists an object $x \in \cC$ together with an isomorphism $y \cong F x$.
We say that $F$ is \textit{fully faithful} if the map
\[
    \Map{\cC} (x, y) \longto \Map{\cD} (F x, F y)
\]
is an equivalence for every pair of objects $x, y \in \cC$.
A \textit{full subcategory} of $\cC$ is an $\infty$-category $\cA$ together with a fully faithful functor $J\colon \cA \to \cC$.
We then say that $J$ is an \textit{inclusion} of $\cA$ into~$\cC$.
Every subset of objects of $\cC$ determines a full subcategory, which is unique up to equivalence.

We say that $F \colon \cC \to \cD$ is \emph{left adjoint} to $G \colon \cD \to \cC$ (or that $G$ is \emph{right adjoint} to $F$) if there exist natural transformations $\mu \colon \Id_\cC \longto GF$ and $\epsilon \colon FG \longto \Id_\cD$ such that the composite transformations
\[
    F
    \overset{F \mu}{\longto}
    FGF
    \overset{\epsilon F}{\longto}
    F,
    \hspace{1.2cm}
    G
    \overset{\mu G}{\longto}
    GFG
    \overset{G \epsilon}{\longto}
    G
\]
are equivalent to ${\rm id}_F$ and ${\rm id}_G$ respectively.

For any $\infty$-category $\cC$, the \textit{opposite $\infty$-category} $\Op{\cC}$ has the same objects as $\cC$ together with $\Map{\Op{\cC}} (x, y) = \Op{\Map{\cC} (y, x)}$, where the opposite of a simplicial set is defined by reversing the indexing of faces and degeneracies \autocite[Definition 1.5.7]{Cisinski2019}.
Consequently, $\Op{(\Op{\cC})} = \cC$.
For any small $\infty$-category $\cA$, we denote $\PSh (\cA) = \Fun{\Op{\cA}}{\Spaces}$ and call it the $\infty$-category of \textit{presheaves} on~$\cA$.

The \textit{slice} $\Slice{\cC}{x}$ of an object $x \in \cC$ is defined as the following pullback:
\begin{equation}%
    \label{slice}
    \begin{tikzcd}[ampersand replacement=\&,cramped]
        {\Slice{\cC}{x}} \&\& {\Fun{\sDelta{1}}{\cC}} \\
        {\cC \simeq \cC \times \Delta^0} \&\& {\cC \times \cC}\rlap{.}
        \arrow[from=1-1, to=1-3]
        \arrow[from=1-1, to=2-1]
        \arrow["{(\Ev_0,\, \Ev_1)}", from=1-3, to=2-3]
        \arrow[""{name=0, anchor=center, inner sep=0}, "{\Id \times x}"', from=2-1, to=2-3]
        \arrow["\lrcorner"{anchor=center, pos=0.125}, draw=none, from=1-1, to=0]
    \end{tikzcd}
\end{equation}
More generally, we define the \textit{relative slice} $\RelSlice{F}{x}$ as the pullback of $\Slice{\cC}{x}$ along a functor $F \colon \cA \to \cC$:
\[
    \begin{tikzcd}[ampersand replacement=\&,cramped]
        {\RelSlice{F}{x}} \& {\Slice{\cC}{x}} \\
        \cA \& \cC\rlap{.}
        \arrow[from=1-1, to=1-2]
        \arrow[from=1-1, to=2-1]
        \arrow["\lrcorner"{anchor=center, pos=0.125}, draw=none, from=1-1, to=2-2]
        \arrow[from=1-2, to=2-2]
        \arrow["F", from=2-1, to=2-2]
    \end{tikzcd}
\]
Dually, the \textit{coslice} $\coSlice{\cC}{x}$ is defined as in~\eqref{slice}, by replacing ${\rm id}\times x$ by $x\times {\rm id}$.
A~relative coslice category $F_{x/}$ along a functor $F \colon \cA \to \cC$ is defined similarly, by taking a pullback of $\coSlice{\cC}{x} \to \cC$.

If $\cK$ is a small $\infty$-category, a \textit{$\cK$-diagram} in an $\infty$-category $\cC$ is a functor $\cK \to \cC$.
For any small $\infty$-category $\cK$ and any object $x \in \cC$, the \emph{constant diagram} $\const x \colon \cK \to \cC$ sends all objects of $\cK$ to $x$ and higher morphisms of $\cK$ to higher identities over~$x$, i.e., the iterated application of the first degeneracy over~$x$.
By applying $\Fun{-}{\cC}$ to the terminal map $\cK \to \sDelta{0}$, we obtain the \textit{diagonal functor} $\const \colon \cC \to \Fun{\cK}{\cC}$, which sends any object $x \in \cC$ to the constant diagram $\const x$, and any morphism $f \colon x \to y$ to a natural transformation $\const f \colon \const x \Rightarrow \const y$
defined by post-composition with~$f$.
The \textit{left cone} $\Cones{K}$ and \textit{right cone} $\coCones{\cK}$ are defined as the following pushouts:
\[
    \begin{tikzcd}[ampersand replacement=\&,cramped]
        {\cK \times \sDelta{0}} \& {\cK \times \sDelta{1}} \\
        {\sDelta{0}} \& {\Cones{\cK}}
        \arrow["{\Id \times 0}", from=1-1, to=1-2]
        \arrow["{\pr_2}"', from=1-1, to=2-1]
        \arrow[from=1-2, to=2-2]
        \arrow[from=2-1, to=2-2]
        \arrow["\lrcorner"{anchor=center, pos=0.125, rotate=180}, draw=none, from=2-2, to=1-1]
    \end{tikzcd}
    \hspace{2cm}
    \begin{tikzcd}[ampersand replacement=\&,cramped]
        {\cK \times \sDelta{0}} \& {\cK \times \sDelta{1}} \\
        {\sDelta{0}} \& {\coCones{\cK}}\rlap{.}
        \arrow["{\Id \times 1}", from=1-1, to=1-2]
        \arrow["{\pr_2}"', from=1-1, to=2-1]
        \arrow[from=1-2, to=2-2]
        \arrow[from=2-1, to=2-2]
        \arrow["\lrcorner"{anchor=center, pos=0.125, rotate=180}, draw=none, from=2-2, to=1-1]
    \end{tikzcd}
\]

For a diagram $D \colon \cK \to \cC$, we define the $\infty$-category of \textit{cones} $\cRelSlice{\cC}{D}$ as the relative slice $\RelSlice{\const}{D}$ along the constant functor.
An object $\alpha \in \cRelSlice{\cC}{D}$ can be viewed as a pair $(x, \alpha)$, where $x \in \cC$ is the \textit{cone point} and $\alpha \colon \const x \Rightarrow D$ is a natural transformation.
By~\autocite[Proposition 1.2.9.2]{Lurie2009}, a cone $\alpha \in \cRelSlice{\cC}{D}$ is equivalent to a functor $\widehat{\alpha} \colon \Cones{\cK} \to \cC$ such that $\restrict{\widehat{\alpha}}{\cK} = D$.

Then, a \textit{limit} for $D$ is a terminal object of $\cRelSlice{\cC}{D}$, which can be viewed as a pair $(\lim D, \ell)$ where $\lim D \in \cC$ and $\ell \colon \const \lim D \Rightarrow D$ is a natural transformation.
By~\autocite[Lemma 4.2.4.3]{Lurie2009}, a limit for $D$ has the following universal property: for all $y \in \cC$, the map
\[
    \Map{\cC} (y, \lim D)
    \overset{\const}{\longrightarrow}
    \Map{\sFun{\cK}{\cC}} (\const y, \const \lim D)
    \xrightarrow{\ell \operatorname{\circ} -}
    \Map{\sFun{\cK}{\cC}} (\const y, D)
\]
is an equivalence of $\infty$-groupoids.
Dually, a \textit{$\cK$-cocone} in $\cC$ is a $\cK$-cone in the opposite $\infty$-category $\Op{\cC}$, and a \textit{colimit} is a limit in~$\Op{\cC}$.

An $\infty$-category $\cC$ is \textit{(co)complete} if, for every small $\infty$-category~$\cK$, each diagram $D \colon \cK \to \cC$ admits a (co)limit.
For example, the $\infty$-category of presheaves over any small $\infty$-category is complete and cocomplete.

Let $\cA$ and $\cK$ be small $\infty$-categories, and $\kappa$ be a regular cardinal.
A $\cK$-diagram, $\cK$-cone or limit $\cK$-cone is \textit{$\kappa$-small} if $\cK$ is $\kappa$-small as an $\infty$-category.
A functor $F \colon \cA \to \Spaces$ is called \textit{$\kappa$-(co)continuous} if it preserves $\kappa$-small (co)limits.
In particular, it is called \textit{finitely (co)continuous} if it is $\aleph_0$-(co)continuous.
We denote by $\Cont_\kappa (\cA)$ the full subcategory of $\Fun{\cA}{\Spaces}$ spanned by all $\kappa$-continuous functors.

Let $F \colon \cC \to \cD$ be a functor of $\infty$-categories which admits a right adjoint $G \colon \cD \to \cC$.
For every small $\infty$-category $\cK$, the functor $F$ preserves colimits and the functor $G$ preserves limits~\autocite[Proposition 5.2.3.5]{Lurie2009}.
Given an object $c \in \cC$, define the \textit{evaluation functor} $\Ev_c \colon \Fun{\cC}{\cD} \to \Fun{\sDelta{0}}{\cD} \simeq \cD$ as the functor resulting from applying $\Fun{-}{\cD}$ to the morphism $c \colon \sDelta{0} \to \cC$ of simplicial sets.
By~\autocite[Proposition 5.1.2.3]{Lurie2009} and its dual,  $\Ev_c \colon \Fun{\cC}{\cD} \to \cD$ preserves limits and colimits for all $c \in \cC$.

Later, we will need both the covariant and the contravariant Yoneda embeddings, and thus we need to fix a notation to distinguish between the two:

\begin{theorem}[Yoneda Lemma]%
    \label{thm:yoneda-lemma}
    Let $\cA$ be a small $\infty$-category.
    There exists a unique functor
    \[
        \Ynd{\cA} \colon \cA \longrightarrow \PSh (\cA)
    \]
    such that $\Ynd{\cA}(x) (y) \simeq \Map{\cA} (y, x)$ for all $x, y \in \cA$.
    Furthermore, $\Ynd{\cA}$ is fully faithful, and, for any object $x \in \cA$
    and any functor $F \colon \Op{\cA} \to \Spaces$, there is a natural equivalence
    \[
        \Map{\PSh (\cA)} (\Ynd{\cA}(x), F)
        \simeq Fx.
    \]
\end{theorem}

We refer to $\Ynd{\cA}$ as the \textit{covariant Yoneda embedding}; see~\autocite[Lemma~5.5.2.1 and Proposition~5.1.3.1]{Lurie2009} for a proof and further details.
Conversely, the covariant Yoneda embedding applied to the opposite of an $\infty$-category $\cA$ yields a unique (and fully faithful) functor
\[
    \coYnd{\cA} \colon \Op{\cA} \longrightarrow \Fun{\cA}{\Spaces}
\]
such that $\coYnd{\cA}(x) (y) \simeq \Map{\cA} (x, y)$ for all $x, y \in \cA$.
We refer to $\coYnd{\cA}$ as the \textit{contravariant Yoneda embedding}. For every functor $G \colon \cA \to \Spaces$, there is a natural equivalence
\[
    \Map{\sFun{\cA}{\Spaces}} (\coYnd{\cA}(x), G)
    \simeq Gx.
\]
By~\autocite[Proposition 5.1.3.2]{Lurie2009}, the covariant Yoneda embedding $\Ynd{\cA}$ preserves all limits that exist in~$\cA$, and the contravariant Yoneda embedding $\coYnd{\cA}$ sends colimits that exist in $\cA$ to limits.

\subsection{Localizations, accessibility and presentability}

A functor $L \colon \cC \to \cD$ between $\infty$\nobreakdash-cat\-egories is a \textit{reflective localization} if it has a fully faithful right adjoint $J \colon \cD \to \cC$.
Hence, $L$ is a reflective localization if and only if $\cD$ embeds as a reflective subcategory into~$\cC$, that is, for every object $x \in \cC$ there exists an object $x' \in \cD$ and a morphism $r \colon x \to J x'$ such that the pre-composition map
\[
    \Map{\cC} (r, z) \colon \Map{\cC} (J x', z) \longrightarrow \Map{\cC} (x, z)
\]
is an equivalence of $\infty$-groupoids for all $z \in \cD$.
The term \textit{reflective} emphasizes a distinction with a more general concept of localization, used in~\autocite{Anel2022}, of a universal functor from $\cC$ inverting a given class of morphisms~$S$, without necessarily being coaugmented.

Let $S$ be a class of morphisms in an $\infty$-category $\cC$.
An object $z \in \cC$ is \textit{$S$-local} if, for every $f\colon x\to y$
in~$S$, there is an equivalence of $\infty$-groupoids induced by composition with $f$:
\[
    f^* \colon \Map{\cC} (y, z) \overset{\simeq}{\longto} \Map{\cC} (x, z).
\]
We denote by $\Loc (\cC, S)$ the full subcategory of $\cC$ spanned by $S$-local objects.
In general, $\Loc (\cC, S)$ need not be reflective.

Let $\cA$, $\cK$ and $\cI$ be small $\infty$-categories, $\cC$ be an $\infty$-category, and $\kappa$ be a regular cardinal.
An $\infty$-category $\cI$ is \textit{$\kappa$-filtered} if it has at least a cocone for any diagram $D \colon \cK \to \cI$ where $\cK$ is $\kappa$-small.
A diagram $F \colon \cI \to \cC$ where $\cI$ is a $\kappa$-filtered $\infty$-category is called a \textit{$\kappa$-filtered diagram}, and a \textit{$\kappa$-filtered colimit} is a colimit over a $\kappa$-filtered diagram.

An~object $x \in \cC$ is \textit{$\kappa$-compact} if $\coYnd{\cC}(x) \colon \cC \to \Spaces$ preserves $\kappa$-filtered colimits.
We denote by $\cC^{\kappa}$ the full subcategory of $\cC$ spanned by $\kappa$-compact objects.
If $\cC$ is a cocomplete $\infty$-category, we say that a class of objects $G \subseteq \Obj(\cC)$ \textit{generates $\cC$ under ($\kappa$-filtered) colimits} if every object in $\cC$ is the ($\kappa$-filtered) colimit of a diagram with objects in $G$.
For example, for every small $\infty$-category $\cA$, the image of the covariant Yoneda embedding $\Ynd{\cA} \colon \cA \to \PSh (\cA)$ generates $\PSh (\cA)$ under colimits; see~\autocite[Corollary 5.1.5.8]{Lurie2009}.

The $\infty$-category $\Ind_\kappa (\cA)$ is defined as the full subcategory of $\PSh (\cA)$ spanned by those presheaves $F \colon \Op{\cA} \to \Spaces$ which classify right fibrations $\widetilde{\cA} \to \cA$ where $\widetilde{\cA}$ is $\kappa$-filtered \autocite[Definition~5.3.5.1]{Lurie2009}.
It is shown in~\autocite[Corollary 5.3.5.4]{Lurie2009} that $\Ind_\kappa (\cA)$ is
the cocompletion of $\cA$ under $\kappa$-filtered colimits, from which it follows that the full subcategory $\Ind_\kappa(\cA) \subseteq \PSh (\cA)$ is stable under $\kappa$-filtered colimits \autocite[Proposition 5.3.5.3]{Lurie2009}.

An $\infty$-category is \textit{$\kappa$-accessible} if it is equivalent to $\Ind_\kappa (\cA)$ for some regular cardinal $\kappa$ and some small $\infty$-category $\cA$.
Furthermore, a functor $F \colon \cC \to \cD$ is \textit{$\kappa$-accessible} if $\cC$ is $\kappa$-accessible and $F$ preserves $\kappa$-filtered colimits.
We say that an $\infty$-category $\cC$ (resp.\ a functor $F \colon \cC \to \cD$) is \textit{accessible} if it is $\kappa$-accessible for some regular cardinal $\kappa$.
When considering a functor between accessible $\infty$-categories, it can be shown~\autocite[Proposition 5.4.7.7]{Lurie2009} that, if it has a left or right adjoint, then it is itself accessible.
As shown by Lurie in~\autocite[Corollary 5.3.5.4]{Lurie2009} and~\autocite[Proposition 5.4.2.2]{Lurie2009}, the following theorem characterizes accessibility, where the third property is equivalent to the classical definition of accessibility for categories~\autocite{Adamek1994,Makkai1989}:

\begin{theorem}[Characterization of accessibility]%
    \label{thm:accessible}
    Let $\cC$ be an $\infty$-category and $\kappa$ be a regular cardinal.
    Then, the following statements are equivalent:
    \begin{enumerate}[label={\rm(\roman*)}]
        \item $\cC$ is a $\kappa$-accessible $\infty$-category.
        \item $\cC$ is equivalent to the full subcategory of $\PSh (\cA)$ spanned by the $\kappa$-filtered colimits of representable presheaves for some small $\infty$-category $\cA$.
        \item $\cC$ is locally small and admits $\kappa$-filtered colimits, the full subcategory $\cC^{\kappa} \subseteq \cC$ of $\kappa$-compact objects is small, and $\cC^{\kappa}$ generates $\cC$ under $\kappa$-filtered colimits.
    \end{enumerate}
\end{theorem}

An $\infty$-category is \textit{$\kappa$-presentable} (resp.\ \textit{presentable}) if it is $\kappa$-accessible (resp.\ accessible) and cocomplete.
We say that a reflective localization $L \colon \cC \to \cD$ is an \textit{accessible} reflective localization if the right adjoint to $L$ is an accessible functor.
For example, any reflective localization between accessible $\infty$-categories is accessible~\autocite[Proposition 5.5.1.2]{Lurie2009}.
As shown by~\textcite{Simpson1999} and~\textcite[Theorem 5.5.1.1]{Lurie2009}, presentability can be characterized in terms of an accessible reflective localization of an $\infty$-category of presheaves:

\begin{theorem}[Characterization of presentability]%
    \label{thm:simpsonlurie}
    Let $\cC$ be an $\infty$-category and $\kappa$ be a regular cardinal.
    Then, the following statements are equivalent:
    \begin{enumerate}[label={\rm(\roman*)}]
        \item $\cC$ is a $\kappa$-presentable $\infty$-category.
        \item $\cC$ is equivalent to $\Ind_\kappa (\cA)$
              for some small $\infty$-category $\cA$ which admits $\kappa$-small colimits.
        \item $\cC$ is equivalent to a $\kappa$-accessible reflective localization of the $\infty$-category of presheaves $\PSh (\cA)$ on some small category $\cA$.
    \end{enumerate}
\end{theorem}

Examples of presentable $\infty$-categories include $\Spaces$, any $\infty$-topos, and the nerve of any presentable category.
If $\cA$ is a small $\infty$-category and $\cC$ is a presentable $\infty$-category, then $\Fun{\cA}{\cC}$ is presentable~\autocite[Proposition 5.5.3.6]{Lurie2009}\label{prop:presentable-fun}.
In particular, $\PSh (\cA)$ and $\Fun{\cA}{\Spaces}$ are presentable for every small $\infty$\nobreakdash-category $\cA$.
Furthermore, every presentable $\infty$-category is not only cocomplete but also complete~\autocite[Corollary 5.5.2.4]{Lurie2009}, and the Adjoint Functor Theorem (see~\autocite[Corollary 5.5.2.9]{Lurie2009} and \autocite{Nguyen2019}) characterizes right (resp.\ left) adjoints of functors between presentable $\infty$-categories as the ones preserving colimits (resp.\ preserving limits and being accessible).

In addition, presentable $\infty$-categories provide a convenient ambient for localization.
As shown in~\autocite[Propositon 2.2.1]{Anel2022}, every reflective localization induces a reflector which inverts a class of morphisms.
Conversely, if we choose a set of morphism $S$ in a presentable $\infty$-category, then there exists a reflective localization which inverts a class $\bar{S}$ containing~$S$; see~\autocite[Proposition 5.5.4.15]{Lurie2009}.

\section{Higher sketches}%
\label{sec:sketches}

Let $\bD$ be any set of diagrams in an $\infty$-category $\cC$.
We say that $\cC$ is $\bD$-(co)complete if all the diagrams of $\bD$ have a (co)limit in $\cC$.

\begin{definition}
    A \textit{sketch} $\Sigma = (\cA, \bL, \bC)$ is a triple consisting of a small $\infty$-category $\cA$, a set of cones $\bL$ in $\cA$ and a set of cocones $\bC$ in $\cA$.
    Given an $\bL$-complete and $\bC$-cocomplete $\infty$-category $\cC$, a \textit{model} of a sketch $\Sigma$ in $\cC$ is a functor $F \colon \cA \to \cC$ that sends cones of $\bL$ to limit cones in $\cC$ and cocones of $\bC$ to colimit cocones in $\cC$.
\end{definition}

Denote by $\Mod (\Sigma, \cC)$ the $\infty$-category of models of $\Sigma$ in $\cC$, and by $\Mod (\Sigma) = \Mod (\Sigma, \Spaces)$ the $\infty$-category of models of $\Sigma$ in $\Spaces$.
A sketch with $\bC = \emptyset$ is called a \textit{limit sketch}, and one with $\bL = \emptyset$ is called a \textit{colimit sketch}.
If $\cA$ is $\bL$-complete and $\bC$-cocomplete, we say that $\Sigma$ is \textit{normal} if all the cones of $\bL$ are limits and all the cocones of $\bC$ are colimits.
Given a regular cardinal $\kappa$, we say that a sketch is \textit{$\kappa$-small} if all the cones of $\bL$ have $\kappa$-small diagrams.
In particular, we say that a sketch is \textit{finite} when it is $\aleph_0$-small.
We say that an $\infty$-category is \textit{(limit) sketchable} if it is equivalent to $\Mod (\Sigma)$ for some (limit) sketch~$\Sigma$.

The condition for a model of a sketch can be rewritten in terms of inverting certain morphisms.
Let $\cC$ be a complete and cocomplete $\infty$-category, $F \colon \cA \to \cC$ be a functor, and $\Sigma = (\cA, \bL, \bC)$ be a sketch.
Consider a cone $\alpha \colon \cK^\triangleleft \to \cA$ of $\bL$, with cone point $x \in \cA$ and diagram $D \colon \cK \to \cA$, and a cocone $\beta \colon \cH^\triangleright \to \cA$ of $\bC$, with cocone point $y \in \cA$ and diagram $E \colon \cH \to \cA$.
Consider the composites
\[
    F \circ \alpha \colon
    \cK^\triangleleft \to \cC
    \qquad\text{and}\qquad
    F \circ \beta \colon
    \cH^\triangleright \to \cC.
\]

Since $\cC$ is complete and cocomplete, we can take the limit and colimit of $F \circ D$ and $F \circ E$ respectively.
By the universal properties of limits and colimits, there are morphisms
\begin{equation}%
    \label{eq:sketch}
    t_\alpha \colon
    F x
    \longto
    \lim_\cK (F \circ D)
    \qquad\text{and}\qquad
    u_\beta \colon
    \colim_\cH (F \circ E)
    \longto
    F y.
\end{equation}
Then, $t_\alpha$ and $u_\beta$ are isomorphisms for all $\alpha \in \bL$ and $\beta \in \bC$ if and only if $F$ is a model of $\Sigma$.

\begin{proposition}%
    \label{prop:modiscomplete}
    Let $\cC$ be a presentable $\infty$-category.
    If $\Sigma$ is a limit sketch, then $\Mod (\Sigma, \cC)$ is complete.
    Dually, if $\Sigma$ is a colimit sketch, then $\Mod (\Sigma, \cC)$ is cocomplete.
\end{proposition}
\begin{proof}
    Let $\Sigma = (\cA, \bL)$ be a limit sketch.
    Since $\cC$ is presentable, $\Fun{\cA}{\cC}$ is complete.
    Thus, for any diagram $F \colon \cI \to \Mod (\Sigma, \cC)$ of models of $\Sigma$, there is a limit $\lim_\cI F \in \Fun{\cA}{\cC}$.
    We want to show that $\lim_\cI F$ is in fact also a model of $\Sigma$.
    By~\eqref{eq:sketch}, $\lim_\cI F \in \Mod (\Sigma, \cC)$ if and only if, for every cone $\alpha \colon \cK^\triangleleft \to \cA$ of $\cL$, with cone point $x \in \cA$ and diagram $D \colon \cK \to \cA$, the morphism
    \begin{equation}%
        \label{eq:previous}
        (\lim_\cI F) \, x
        \longto
        \lim_\cK ((\lim_\cI F) \circ D)
    \end{equation}
    is an isomorphism.
    Since $\cK$ is small and $\cC$ is complete, $(- \circ D)$ is a right adjoint and preserves limits.
    Using that the evaluation functor also preserves limits, and that limits commute with limits, \eqref{eq:previous} is an isomorphism if and only if
    \begin{equation}%
        \label{eq:2nd_previous}
        \lim_{i \in \cI} (F(i) x)
        \longto
        \lim_{i \in \cI} \, \lim_\cK ( F(i) \circ D)
    \end{equation}
    is an isomorphism in~$\cC$.
    Thus, since each $F(i)$ is a model of $\Sigma$, the equivalences in~\eqref{eq:sketch} hold for every $F(i)$, and~\eqref{eq:2nd_previous} is an isomorphism, as we wanted to show.
    If $\Sigma$ is a colimit sketch, the same argument follows using the fact that $\cC$ is cocomplete.
\end{proof}

\begin{example}[Morphisms]
    Let $\Sigma = (\sDelta{1}, \emptyset)$ be the trivial limit sketch over $\sDelta{1}$.
    A model $F \colon \sDelta{1} \to \cC$ of $\Sigma$ in any $\infty$-category $\cC$ exhibits a morphism of $\cC$.
    Therefore, $\Mod(\Sigma, \cC) = \Fun{\sDelta{1}}{ \cC}$ is the $\infty$-category of morphisms in~$\cC$.
    The same construction can be carried out with any small $\infty$-category~$\cA$; hence $\Mod((\cA, \emptyset), \cC) = \Fun{\cA}{\cC}$ is limit-sketchable.
\end{example}

\begin{example}[Pointed objects]
    Let $f \colon 0 \to 1$ be the generating morphism of $\sDelta{1}$, and let $\cC$ be an $\infty$-category with terminal object $1_\cC$.
    Let $\Sigma = (\sDelta{1}, \bL)$ be the limit sketch with set of cones $\bL$ consisting of the empty diagram $D \colon \emptyset \to \sDelta{1}$, cone point $0$, and the unique natural transformation $\alpha \colon \const 0 \Rightarrow D$.
    A model $F \colon \sDelta{1} \to \cC$ of $\Sigma$ in $\cC$ sends $f$ to a morphism $F(f) \colon F(0) \to F(1)$, and it also sends the only cone of $\bL$ to a limit cone of the diagram $F \circ D \colon \emptyset \to \cC$.
    It follows that $F(0) \cong \lim (F \circ D) \cong 1_\cC$.
    Therefore, each model $F$ exhibits an object $F(1) \in \cC$ as a pointed object $F(f) \colon 1_\cC \to F(1)$ of $\cC$, and $\Mod(\Sigma, \cC)$ can be viewed as the $\infty$-category of pointed objects of~$\cC$.
    In particular, $\Mod(\Sigma,\cS)$ is the $\infty$-category of pointed spaces.
\end{example}

\begin{example}[Pullback diagrams]
    Let $\cA$ be the nerve of the small category generated by a commutative square
    \[
        \begin{tikzcd}[ampersand replacement=\&]
            3 \& 2 \\
            1 \& 0\rlap{.}
            \arrow[from=1-1, to=2-1]
            \arrow[from=1-1, to=1-2]
            \arrow[from=1-2, to=2-2]
            \arrow[from=2-1, to=2-2]
        \end{tikzcd}
    \]
    Consider a limit sketch $\Sigma = (\cA, \bL)$ with set of cones $\bL$ consisting of the inclusion diagram $D \colon \{1 \to 0 \leftarrow 2\} \to \cA$, cone point $3 \in \cA$, and the natural transformation $\alpha \colon \const 3 \Rightarrow D$.
    A~model $F \colon \cA \to \cC$ in a complete $\infty$-category $\cC$ for the sketch $\Sigma$ sends
    \[
        \begin{tikzcd}[ampersand replacement=\&,cramped]
            3 \& 2 \\
            1 \& 0
            \arrow[from=1-1, to=2-1]
            \arrow[from=1-1, to=1-2]
            \arrow[from=1-2, to=2-2]
            \arrow[from=2-1, to=2-2]
        \end{tikzcd}
        \begin{tikzcd}[ampersand replacement=\&,cramped]
            {} \&\& {}
            \arrow["F", shorten <=10pt, shorten >=10pt, maps to, from=1-1, to=1-3]
        \end{tikzcd}
        \begin{tikzcd}[ampersand replacement=\&,cramped]
            {x_3} \& {x_2} \\
            {x_1} \& {x_0}
            \arrow[from=2-1, to=2-2]
            \arrow[from=1-2, to=2-2]
            \arrow[from=1-1, to=1-2]
            \arrow[from=1-1, to=2-1]
            \arrow["\lrcorner"{anchor=center, pos=0.125}, draw=none, from=1-1, to=2-2]
        \end{tikzcd}
    \]
    where the right-hand commutative square is a pullback diagram.
    Therefore, each model of $\Sigma$ corresponds to a pullback diagram in~$\cC$.
\end{example}

\begin{example}[Pre-spectrum and spectrum objects]
    Let $\cA$ be the nerve of the small category with objects $\N \sqcup (\N \times \{0, 1\})$ and generating morphisms $f_{i,j} \colon i \to (i+1, j)$ and $g_{i,j} \colon (i, j) \to i$
    for every $i \in \N$ and $j \in \{0, 1\}$, i.e., the category of the following shape:
    
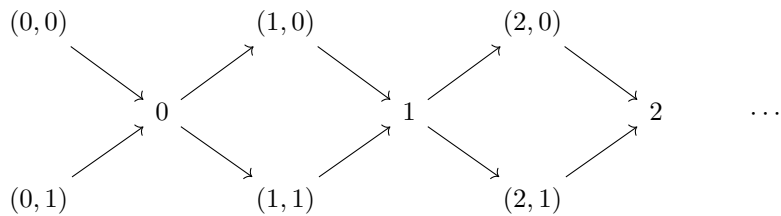
\[
        \begin{tikzcd}[ampersand replacement=\&]
            {(0,0)} \&\& {(1,0)} \&\& {(2,0)} \\
            \& 0 \&\& 1 \&\& 2 \& \cdots \\
            {(0,1)} \&\& {(1,1)} \&\& {(2,1)}
            \arrow[from=3-1, to=2-2]
            \arrow[from=1-1, to=2-2]
            \arrow[from=2-2, to=1-3]
            \arrow[from=1-3, to=2-4]
            \arrow[from=2-2, to=3-3]
            \arrow[from=3-3, to=2-4]
            \arrow[from=2-4, to=1-5]
            \arrow[from=1-5, to=2-6]
            \arrow[from=2-4, to=3-5]
            \arrow[from=3-5, to=2-6]
        \end{tikzcd}
    \]

    Consider the limit sketch $\Sigma = (\cA, \bL)$ with set of cones $\bL$ consisting of, for each $i \in \N$ and $j \in \{0,1\}$, the unique cone of the empty diagram and cone point $(i,j) \in \cA$.
    A~model $F \colon \cA \to \cC$ in a complete $\infty$-category $\cC$ for the sketch $\Sigma$ is a diagram
    
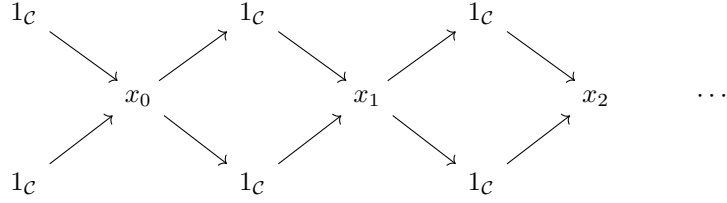
\[
        \begin{tikzcd}[ampersand replacement=\&]
            {1_\cC} \&\& {1_\cC} \&\& {1_\cC} \\
            \& {x_0} \&\& {x_1} \&\& {x_2} \& \cdots \\
            {1_\cC} \&\& {1_\cC} \&\& {1_\cC}
            \arrow[from=3-1, to=2-2]
            \arrow[from=1-1, to=2-2]
            \arrow[from=2-2, to=1-3]
            \arrow[from=1-3, to=2-4]
            \arrow[from=2-2, to=3-3]
            \arrow[from=3-3, to=2-4]
            \arrow[from=2-4, to=1-5]
            \arrow[from=1-5, to=2-6]
            \arrow[from=2-4, to=3-5]
            \arrow[from=3-5, to=2-6]
        \end{tikzcd}
    \]
    where each $(i,j)$ is replaced by the terminal object $1_\cC$ of $\cC$, and a sequence of objects $x_n \in \cC$ is selected.
    Giving a model of $\Sigma$ amounts to choosing pointed objects $1_\cC \to x_n$ for all $n \in \N$ and morphisms $x_n \to \Omega x_{n+1}$ (by the universal property of the pullback):
    
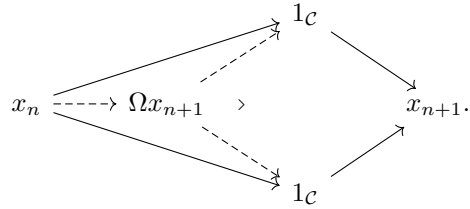
\[
        \begin{tikzcd}[ampersand replacement=\&]
            \&\& 1_\cC \\
            {x_n} \& {\Omega x_{n+1}} \&\& {x_{n+1}.} \\
            \&\& 1_\cC
            \arrow[dashed, from=2-2, to=3-3]
            \arrow[dashed, from=2-2, to=1-3]
            \arrow[dashed, from=2-1, to=2-2]
            \arrow[from=2-1, to=3-3]
            \arrow[from=2-1, to=1-3]
            \arrow[from=3-3, to=2-4]
            \arrow[from=1-3, to=2-4]
            \arrow["\lrcorner"{anchor=center, pos=0.125, rotate=45}, draw=none, from=2-2, to=2-4]
        \end{tikzcd}
    \]
    Hence, each model of $\Sigma$ is, by definition, a \textit{pre-spectrum object}, and $\Mod(\Sigma, \cC)$ is the $\infty$-category of pre-spectrum objects in~$\cC$.

    If we want to obtain spectrum objects, we need to add more cones to~$\Sigma$.
    Consider a limit sketch $\Sigma' =  (\cA, \bL \sqcup \bL')$ where $\bL'$ consists of, for each $n \in \N$, a cone
    
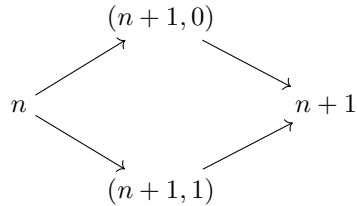
\[
        \begin{tikzcd}[ampersand replacement=\&,cramped]
            \& {(n+1,0)} \\
            {n} \&\& n+1 \\
            \& {(n+1, 1)}
            \arrow[from=1-2, to=2-3]
            \arrow[from=2-1, to=1-2]
            \arrow[from=2-1, to=3-2]
            \arrow[from=3-2, to=2-3]
        \end{tikzcd}
    \]
    where $n$ is the cone point and $(n+1, 0) \to n+1 \leftarrow (n+1, 1)$ is the corresponding diagram.
    A~model for $\Sigma'$ is a pre-spectrum object in $\cC$ such that
    \[
        x_n \cong \text{ pullback of } \{1_\cC \to x_{n+1} \leftarrow 1_\cC\} \cong \Omega x_{n+1}.
    \]
    Thus, a model for $\Sigma'$ is a spectrum object, and  $\Mod (\Sigma', \cC)$ is the $\infty$-category of spectrum objects in~$\cC$.
    In particular, $\Mod (\Sigma',\cS)$ is the $\infty$-category of spectra $\Spectra$.
\end{example}

In the following examples, we view categories such as the simplex category $\Deltabf$, the category $\Gammabf$ of finite pointed sets and pointed maps, or the tree category $\Omegabf$ defined in [8, \S~3.2] as $\infty$-categories by passing to their respective nerves.

\begin{example}[Pre-category objects and Segal spaces]%
    \label{ex:cat-object}
    Let \(\alpha_n\) be the following cone over \(\Op{\Delta}\) with cone point $[n]$:
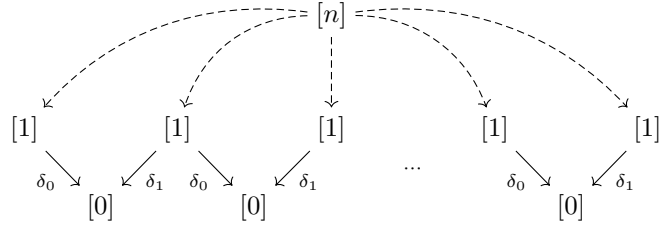
    \[
        \begin{tikzcd}[ampersand replacement=\&,cramped,column sep=small,row sep=scriptsize]
            \&\&\&\& {[n]} \\
            \\
            {[1]} \&\& {[1]} \&\& {[1]} \&[2em]\& {[1]} \&\& {[1]} \\
            \& {[0]} \&\& {[0]} \&\&\&\& {[0]}
            \arrow[curve={height=23pt}, dashed, from=1-5, to=3-1]
            \arrow[curve={height=18pt}, dashed, from=1-5, to=3-3]
            \arrow[dashed, from=1-5, to=3-5]
            \arrow[curve={height=-18pt}, dashed, from=1-5, to=3-7]
            \arrow[curve={height=-23pt}, dashed, from=1-5, to=3-9]
            \arrow["{\dFace{1}{0}}"', from=3-1, to=4-2]
            \arrow["{\dFace{1}{1}}", from=3-3, to=4-2]
            \arrow["{\dFace{1}{0}}"', from=3-3, to=4-4]
            \arrow[""{name=0, anchor=center, inner sep=0}, "{\dFace{1}{1}}", from=3-5, to=4-4]
            \arrow[""{name=1, anchor=center, inner sep=0}, "{\dFace{1}{0}}"', from=3-7, to=4-8]
            \arrow["{\dFace{1}{1}}", from=3-9, to=4-8]
            \arrow["\cdots"{description}, draw=none, from=0, to=1]
        \end{tikzcd}
    \]

    Consider the limit sketch \(\Sigma = (\Op{\Deltabf}, {\{\alpha_n\}}_{n \in \N})\).
    A model $F \colon \Op{\Deltabf} \to \cC$ of $\Sigma$ in a complete $\infty$-cat\-egory $\cC$ is a simplicial object in $\cC$ equipped with isomorphisms
    \[
        F_n \overset{\simeq}{\longto} F_1 \times_{F_0} F_1 \times_{F_0} \cdots \times_{F_0} F_1
    \]
    for all~$n$. Therefore, $\Mod(\Sigma, \cC)$ is the $\infty$-category of \textit{pre-category} objects in~$\cC$.
    If $\cC=\Spaces$, then $\Mod(\Sigma)$ is the $\infty$-category of Segal spaces.
\end{example}

\begin{example}[Univalent category objects and complete Segal spaces]%
    \label{ex:comsegal}
    Let $\cC$ be a complete $\infty$-category and $\Sigma_{\rm C} = (\Op{\Deltabf}, \bL_{\rm C})$ be the sketch of~\autoref{ex:cat-object}.
    By the characterization found in~\autocite[\S\,5.5]{Rezk2000} and~\autocite[Proposition 6.4]{Rezk2000}, a Segal space $F$ is a complete Segal space if and only if the following is a pullback square in $\Spaces$:
    \begin{equation}\label{eq:ex-CSS-lim}\begin{tikzcd}[ampersand replacement=\&]
            {F_0} \& {F_3} \\
            {F_1} \& {F_1 \times_{F_0}^{d_1,\,d_1} F_1 \times_{F_0}^{d_0,\,d_0} F_1}
            \arrow[from=1-1, to=1-2]
            \arrow[from=1-1, to=2-1]
            \arrow["f", from=2-1, to=2-2]
            \arrow["g", from=1-2, to=2-2]
            \arrow["\lrcorner"{anchor=center, pos=0.125}, draw=none, from=1-1, to=2-2]
        \end{tikzcd}\end{equation}
    where $f = (s_0d_0,\,{\rm id}_{F_1},\,s_0d_1)$ and $g = (d_1d_3,\,d_0d_3,\,d_1d_0)$.

    Define a sketch $\Sigma = (\Op{\Deltabf}, \bL)$ where $\bL$ is the union of $\bL_{\rm C}$ with the cone represented by the following diagram:
    
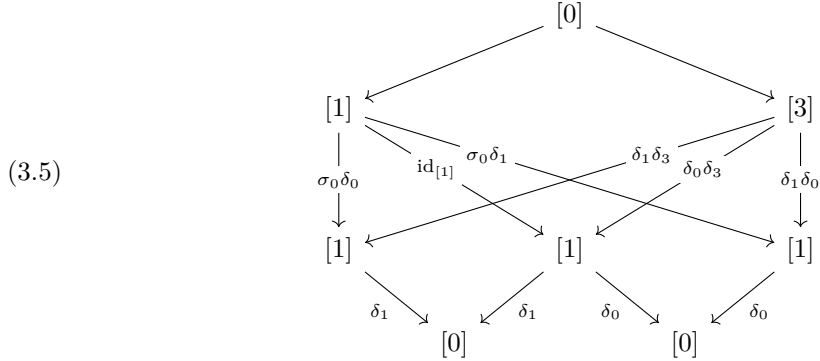
\begin{equation}\label{eq:ex-CSS-diag}\begin{tikzcd}[ampersand replacement=\&]
            \&\& {[0]} \\
            {[1]} \&\&\&\& {[3]} \\
            \\
            {[1]} \&\& {[1]} \&\& {[1]} \\
            \& {[0]} \&\& {[0]}
            \arrow[from=1-3, to=2-5]
            \arrow[from=1-3, to=2-1]
            \arrow["{\dFace{1}{1}}"', from=4-1, to=5-2]
            \arrow["{\dFace{1}{1}}", from=4-3, to=5-2]
            \arrow["{\dDege{1}{0}\dFace{1}{0}}"{description}, from=2-1, to=4-1]
            \arrow["{\Id_{[1]}}"{description, pos=0.4}, from=2-1, to=4-3]
            \arrow["{\dDege{1}{0}\dFace{1}{1}}"{description, pos=0.3}, from=2-1, to=4-5]
            \arrow["{\dFace{3}{1}\dFace{3}{3}}"{description, pos=0.3}, from=2-5, to=4-1]
            \arrow["{\dFace{3}{0}\dFace{3}{3}}"{description, pos=0.4}, from=2-5, to=4-3]
            \arrow["{\dFace{3}{1}\dFace{3}{0}}"{description}, from=2-5, to=4-5]
            \arrow["{\dFace{1}{0}}"', from=4-3, to=5-4]
            \arrow["{\dFace{1}{0}}", from=4-5, to=5-4]
        \end{tikzcd}\end{equation}
    with cone point $[0]$.
    A model $F \colon \Op{\Deltabf} \to \cC$ of $\Sigma$ exhibits a pre-category object in $\cC$ and the image of~\eqref{eq:ex-CSS-diag} is a limit cone, which is equivalent to the pullback square~\eqref{eq:ex-CSS-lim}.
    Therefore, $\Mod (\Sigma, \cC)$ is the $\infty$-category of \emph{univalent category objects} in~$\cC$.
    In the particular case when $\cC = \Spaces$, we have that $\Mod (\Sigma)$ is the $\infty$-category of complete Segal spaces.
\end{example}

\begin{example}[Monoid objects and $A_\infty$-spaces/rings]%
    \label{ex:monoid-object}
    Let $\cC$ be a complete $\infty$-category and $\Sigma_{\rm C} = (\Op{\Deltabf}, \bL_{\rm C})$ be the sketch of~\autoref{ex:cat-object}.
    Define a sketch $\Sigma = (\Op{\Deltabf}, \bL)$ where $\bL$ is the union of $\bL_{\rm C}$ and a cone with empty diagram and cone point~$[0]$.
    Each model of $\Sigma$ is a pre-category object $F$ in $\cC$ such that $F_0 \cong 1_\cC$.
    Hence, $\Mod (\Sigma, \cC)$ is the $\infty$-category of monoid objects in~$\cC$.
    In the case when $\cC = \Spaces$, we have that $\Mod (\Sigma)$ is the $\infty$-category of $A_\infty$-spaces.
    If $\cC = \Spectra$, then $\Mod (\Sigma, \Spectra)$ is the $\infty$-category of $A_\infty$-ring spectra.
\end{example}

\begin{example}[Groupoid objects]%
    \label{ex:groupoids}
    Let $\cC$ be a complete $\infty$-category and $\Sigma = (\Op{\Deltabf}, \bL_{\rm C})$ be the sketch of~\autoref{ex:cat-object}.
    Define a sketch $\Sigma = (\Op{\Deltabf}, \bL)$ where $\bL$ is the union of $\bL_{\rm C}$ with a diagram
    \[
        D \colon \{1 \to 0 \leftarrow 2\} \longto   \Op{\Deltabf}
    \]
    sending $1 \to 0 \leftarrow 2$ to $[1]          \xrightarrow{\dFace{1}{0}} [0] \xleftarrow{\dFace{1}{0}} [1]$
    and a natural transformation $\alpha$ with cone point $[2]$ defining the following commutative square:
    \[
        \begin{tikzcd}[ampersand replacement=\&]
            {[2]} \& {[1]} \\
            {[1]} \& {[0]\rlap{.}}
            \arrow["{\dFace{1}{0}}"', from=2-1, to=2-2]
            \arrow["{\dFace{1}{0}}", from=1-2, to=2-2]
            \arrow["{\dFace{2}{0}}"', from=1-1, to=2-1]
            \arrow["{\dFace{2}{1}}", from=1-1, to=1-2]
        \end{tikzcd}
    \]
    A model of $\Sigma$ defines a pre-category object and sends these squares to pullback squares.
    Therefore, $\Mod (\Sigma, \cC)$ is the $\infty$-category of groupoid objects in~$\cC$.
\end{example}

\begin{example}[Group objects and grouplike $A_\infty$-spaces]
    Following the construction used in~\autoref{ex:monoid-object} but replacing the sketch of pre-categories with the one of groupoids, it follows that $\Mod (\Sigma, \cC)$ is the $\infty$-category of group objects in~$\cC$.
    In the particular case when $\cC = \Spaces$, we have that $\Mod (\Sigma)$ is the $\infty$-category of grouplike $A_\infty$-spaces.
\end{example}

\begin{example}[Commutative monoid objects and $E_\infty$-spaces/rings]%
    \label{ex:Einfspaces}
    Let $\Gammabf$ be the category of finite pointed sets and pointed maps, where every object is isomorphic to a set $[n]$ pointed by $0 \in [n]$.
    For each $1 \leq k \leq n$, there is a pointed map $\delta_k \colon [n] \to [1]$ defined by
    \[
        \delta_k (i) =
        \begin{cases}
            1 & \text{ if } i = k,    \\
            0 & \text{ if } i \neq k.
        \end{cases}
    \]
    A $\Gammabf$-object in an $\infty$-category $\cC$ is a functor $E \colon \Gammabf \to \cC$.
    If $\cC$ has finite products, we can take the product of the morphisms $E(\delta_k) \colon E_n \to E_1$, which we denote by
    \[
        p_n \colon E_n \longrightarrow \prod_{k=1}^n E_1.
    \]
    By definition, $E$ is a commutative monoid object if $p_n$ is invertible for every $n \geq 0$.

    Consider a sketch $\Sigma = (\Gammabf, \bL)$,
    where the set of cones $\bL$ consists of, for each $n \in \N$, a diagram
    \[
        D_n \colon \bigsqcup_{k=1}^n \{k\} \longto \Gammabf
    \]
    sending $k \mapsto [1]$ for all~$k$, cone point $[n] \in \Gammabf$, and the natural transformation $\delta_\bullet^n \colon \const [n] \Rightarrow D_n$ induced by $\delta_k$ at each object~$k$.
    Therefore, $\Mod (\Sigma, \cC)$ is the $\infty$-category of commutative monoid objects in~$\cC$.
    If $\cC = \Spaces$, then $\Mod (\Sigma)$ is the $\infty$-category of $E_\infty$-spaces,
    and if $\cC = \Spectra$, then $\Mod (\Sigma, \Spectra)$ is the $\infty$-category of $E_\infty$-ring spectra.
\end{example}

\begin{example}[Abelian group objects and infinite loop spaces]
    Let $\sSet$ denote the category of simplicial sets, and $\cC$ be a complete $\infty$-category.
    Consider the functor
    \[
        i \colon \Op{\Deltabf} \longto \Gammabf
    \]
    sending $[n]\mapsto \Hom{\sSet_*} (\sDelta{n}_+, S^1)$,
    where $S^1$ is the pointed simplicial circle $\sDelta{1} / \partial\sDelta{1}$.
    Since $\cC$ is complete, the map
    \[
        i^* \colon \Fun{\Gammabf}{\cC} \longrightarrow \Fun{\Op{\Deltabf}}{\cC}
    \]
    sends every commutative monoid object $E \in \Mod(S, \cC) \subseteq \Fun{\Gammabf}{\cC}$ to its underlying monoid $i^* E \colon \Op{\Deltabf} \to \cC$.
    We say that $E$ is an \textit{abelian group object} if $i^* E$ is a group.

    Let $\Sigma_{\rm cM} = (\Gammabf, \bL_{\rm cM})$ be the sketch of~\autoref{ex:Einfspaces}, and $\alpha_{\rm Gpd}$ be the cone added in~\autoref{ex:groupoids}.
    Define a sketch $\Sigma = (\Gammabf, \bL_{\rm cM} \sqcup \{i \circ \alpha_{\rm Gpd}\})$, where the cone $i \circ \alpha_{\rm Gpd}$ with cone point $i[2]$ corresponds to the following commutative square:
    \[
        \begin{tikzcd}[ampersand replacement=\&]
            {i [2]} \& {i [1]} \\
            {i [1]} \& {i [0]\rlap{.}}
            \arrow["{i (\dFace{1}{0})}"', from=2-1, to=2-2]
            \arrow["{i (\dFace{1}{0})}", from=1-2, to=2-2]
            \arrow["{i (\dFace{2}{0})}"', from=1-1, to=2-1]
            \arrow["{i (\dFace{2}{1})}", from=1-1, to=1-2]
        \end{tikzcd}
    \]
    A model of $\Sigma$ defines an abelian group object by sending these squares to pullback squares.
    Therefore, $\Mod (\Sigma, \cC)$ is the $\infty$-category of abelian group objects in~$\cC$.
    In the particular case when $\cC = \Spaces$, we have that $\Mod (\Sigma)$ is the $\infty$-category of infinite loop spaces.
\end{example}

\begin{example}[Dendroidal Segal spaces]%
    \label{ex:dendroidal}
    Let $\Omegabf$ be the tree category of Moerdijk--Weiss~\autocite[\S\,3.2]{Heuts2022}.
    Given two trees $T_1$ and $T_2$ sharing an edge $e$ which is a leaf of~$T_1$ and the root of~$T_2$, the \emph{grafting} $T_1 \cup_e T_2$ is the pushout of $T_1$ and $T_2$ along the common edge~$e$.

    Define a limit sketch $\Sigma = (\Op{\Omegabf}, \bL)$ with the set $\bL$ consisting of, for each tree $T \in \Omegabf$ and each decomposition of $T$ as a grafting of subtrees $T = T_1 \circ_e T_2$, a cone with cone point $T$ represented by the following pushout in $\Omegabf$:
    \[
        \begin{tikzcd}[ampersand replacement=\&,cramped]
            \eta \& {T_1} \\
            {T_2} \& T\rlap{.}
            \arrow["e"', from=1-1, to=2-1]
            \arrow[dashed, from=2-1, to=2-2]
            \arrow[dashed, from=1-2, to=2-2]
            \arrow["e", from=1-1, to=1-2]
            \arrow["\ulcorner"{anchor=center, pos=0.175}, draw=none, from=2-2, to=1-1]
        \end{tikzcd}
    \]
    A model for the sketch $\Sigma$ is equivalent to a dendroidal space $X \colon \Op{\Omegabf} \to \Spaces$ such that the squares of the form
    \[
        \begin{tikzcd}[ampersand replacement=\&]
            {X(T)} \& {X(T_1)} \\
            {X(T_2)} \& {X(\eta)}
            \arrow["{e^*}"', from=2-1, to=2-2]
            \arrow[from=1-1, to=2-1]
            \arrow[from=1-1, to=1-2]
            \arrow["\lrcorner"{anchor=center, pos=0.125}, draw=none, from=1-1, to=2-2]
            \arrow["{e^*}", from=1-2, to=2-2]
        \end{tikzcd}
    \]
    are pullbacks for any tree $T$ and any decomposition of $T$ as a grafting of subtrees $T = T_1 \circ_e T_2$.
    By~\autocite[Lemma 12.7]{Heuts2022}, this condition is equivalent to claiming that $X$ is a dendroidal Segal space, and hence $\Mod(\Sigma)$ is the $\infty$-category of dendroidal Segal spaces.
\end{example}

\begin{example}[Complete dendroidal Segal spaces]%
    \label{ex:oper-object}
    Consider the inclusion $j \colon \Op{\Deltabf} \to \Op{\Omegabf}$ sending $[n]$ to the linear tree $L_n$.
    The induced map
    \[
        j^* \colon \Fun{\Op{\Omegabf}}{\Spaces} \longrightarrow \Fun{\Op{\Deltabf}}{\Spaces}
    \]
    sends every dendroidal space $X$ to its underlying simplicial space $j^* X$.
    By~\autocite[Remark 12.15]{Heuts2022}, a dendroidal Segal space $X \colon \Op{\Omegabf} \to \Spaces$ is complete if and only if its underlying simplicial space $j^* X$ is complete.

    Let $\Sigma_{{\rm dS}} = (\Op{\Omegabf}, \bL_{{\rm dS}})$ be the sketch of~\autoref{ex:dendroidal}, which models dendroidal Segal spaces, and let $(D, [0], \alpha)$ be the cone added in~\autoref{ex:comsegal}, which models the completeness condition on simplicial spaces.
    Define a sketch $\Sigma = (\Op{\Omegabf}, \bL)$ where $\bL$ is the union of $\bL_{{\rm dS}}$ and a cone consisting of a diagram $j \circ D$ and a natural transformation $j \circ \alpha$ with cone point $j [0] = L_0 \in \Op{\Omegabf}$.
    A~model of $\Sigma$ is a dendroidal space $X \colon \Op{\Omegabf} \to \Spaces$ such that the map
    \[
        (j^* X) [0] = X j [0] \longto \lim (X \circ j \circ D) = \lim ((j^* X) \circ D)
    \]
    is an isomorphism.
    This condition is equivalent to imposing that the underlying simplicial space $j^* X$ be complete, according to \autoref{ex:comsegal}.
    Hence, $X$ is a model of $\Sigma$ if and only if it is a complete dendroidal Segal space, and $\Mod (\Sigma)$ is the $\infty$-category of complete dendroidal Segal spaces.
\end{example}

In all the examples of limit sketches given so far, the small $\infty$-category associated to the limit sketch is in fact the nerve of a small category.
The following example of higher sheaves has any small $\infty$-category as the base of the corresponding limit sketch:

\begin{example}[Higher sheaves]\label{ex:sheaves}
    Let $\cA$ be a small $\infty$-category, and let $\Slice{\cA}{x}$ denote the slice category over an object $x \in \cA$.
    A \textit{sieve} on an object $x \in \cA$ is a full subcategory $\cD_x \subseteq \Slice{\cA}{x}$ closed under precomposition with morphisms in $\Slice{\cA}{x}$.
    For $S$ a sieve on $x \in \cA$ and $f \colon y \to x$ a morphism, the pullback sieve $f^* S$ on $y$ is the sieve spanned by the morphisms into $y$ that become equivalent to a morphism in $S$ after composition with~$f$.

    A \textit{Grothendieck topology} $\cT$ on an $\infty$-category $\cA$, as defined in~\autocite[\S\,6.2.2]{Lurie2009}, is an assignment to each object $x \in \cA$ of a collection $\cT_x$ of sieves on $x$, called \textit{covering sieves}, such that:
    \begin{enumerate}
        \item For each $x \in \cA$, the trivial sieve $\Slice{\cA}{x} \subseteq \Slice{\cA}{x}$ on $x$ is a covering sieve.
        \item If $S$ is a covering sieve on $x$ and $f \colon y \to x$ is a morphism, then the pullback sieve $f^* S$ is a covering sieve on~$y$.
        \item For a covering sieve $S$ on $x$ and any sieve $R$ on~$x$, if the pullback sieve $f^* R$ is covering for every $f \in S$, then $R$ itself is covering.
    \end{enumerate}

    By~\autocite[Proposition 6.2.2.5]{Lurie2009}, there is a natural bijection between sieves on $x$ in $\cA$ and equivalence classes of monomorphisms $U \to \Ynd{\cA}(x)$ in $\PSh (\cA)$, where $\Ynd{\cA}$ is the Yoneda functor, as in Theorem~\ref{thm:yoneda-lemma}, and a morphism $U\to V$ is a \emph{monomorphism} if it is a $(-1)$-truncated object of $\Slice{\PSh (\cA)}{V}$.

    Let $S(\cT)$ be the class of monomorphisms in $\PSh (\cA)$ corresponding to the covering sieves of~$\cT$.
    A~presheaf $F \in \PSh (\cA)$ is a \textit{sheaf} with Grothendieck topology $\cT$ if it is an $S(\cT)$-local object, i.e., if for every morphism $f \colon U \to \Ynd{\cA}(x)$ in~$S(\cT)$, the map
    \[
        F x
        \simeq \Map{\PSh (\cA)}(\Ynd{\cA}(x), F)
        \longto \Map{\PSh (\cA)}(U, F)
    \]
    is an equivalence.

    Recall that any presheaf $U \in \PSh (\cA)$ can be expressed as a canonical colimit $\colim \Ynd{\cA} \circ \pi$, where $\pi \colon \hRelSlice{\cA}{U} \to \cA$ is the associated forgetful functor.
    Let $\Sigma = (\Op{\cA}, \bL)$ be a limit sketch where the set $\bL$ consists of, for each covering sieve with corresponding monomorphism $f \colon U \to \Ynd{\cA} (x)$, a cone with diagram $\Op{\pi} \colon \Op{\hRelSlice{\cA}{U}} \to \Op{\cA}$, cone point $x \in \cA$ and natural transformation given by~$\Op{f}$.
    A~model of $\Sigma$ is a presheaf $F \colon \Op{\cA} \to \Spaces$ such that
    \[\begin{aligned}
            Fx
            \longto \lim_{\Op{(\hRelSlice{\cA}{U})}} F \circ \Op{\pi}
             & \simeq \lim_{\hRelSlice{\cA}{U}} \Map{\PSh (\cA)}(\coYnd{\cA} \circ \pi, \, F)         \\
             & \simeq \Map{\PSh (\cA)}\Big(\colim_{\hRelSlice{\cA}{U}} \Ynd{\cA} \circ \pi, \, F\Big) \\
             & \simeq \Map{\PSh (\cA)}(U, F)
        \end{aligned}\]
    is an equivalence for every $f \colon U \to \Ynd{\cA} (x)$, i.e., $F$ is a sheaf with Grothendieck topology $\cT$.
    Hence, $\Mod(\Sigma)$ is equivalent to the $\infty$-category $\Sh(\cA, \cT)$ of sheaves on $\cA$ with Grothendieck topology~$\cT$.

    This is not the only way to present the $\infty$-category of sheaves by means of a limit sketch.
    If we assume that $\cA$ has pullbacks, we can formulate the sheaf condition in terms of a limit based on the \v{C}ech nerve.
    We will need the following lemma:

    \begin{lemma}\label{lem:sheaf}
        Let ${\{u_i \to x\}}_{i \in I}$ be a family of morphisms of $\cA$ that generate the covering sieve corresponding to a monomorphism $\eta \colon U \to \Ynd{\cA}(x)$, and let $U_\bullet \colon \Op{\Delta} \to \PSh (\cA)$ be the underlying simplicial object of the \v{C}ech nerve of the induced map $\coprod_{i \in I} \Ynd{\cA}(u_i) \to \Ynd{\cA}(x)$.
        Then, a presheaf $F$ is $\eta$-local if and only if the induced map
        \[
            Fx \longto \lim_{\Op{\Delta}}\, \Map{\PSh (\cA)} (U_\bullet, F)
        \]
        is an equivalence.
        If $\cA$ has pullbacks, the diagram $U_\bullet$ can be decomposed into some diagram $\widetilde{U}_\bullet \colon \Op{\Delta} \to \cA$ and the Yoneda embedding $\Ynd{\cA} \colon \cA \to \PSh (\cA)$, and then a presheaf $F$ is $\eta$-local if and only if the induced map
        \[
            Fx \longto \lim_{\Delta} \, F \circ \Op{\widetilde{U}_\bullet}
        \]
        is an equivalence.
    \end{lemma}
    \begin{proof}
        Let $I$ be a set, and ${\{u_i \to x\}}_{i \in I}$ be a family of morphisms of $\cC$ that generate a covering sieve corresponding to a monomorphism $\eta \colon U \to \Ynd{\cA}(x)$.
        By~\autocite[Lemma 6.2.3.18]{Lurie2009}, $f \colon U \to \Ynd{\cA}(x)$ can be identified with the $(-1)$-truncation of the induced map $\coprod_{i \in I} \Ynd{\cA}[u_i] \to \Ynd{\cA}(x)$ in $\Slice{\PSh (\cA)}{\Ynd{\cA}(x)}$.
        Since $\PSh (\cA)$ is an $\infty$-topos, by~\autocite[Proposition 6.2.3.4]{Lurie2009}, the $(-1)$-truncation of a map $p \colon V \to \Ynd{\cA}(x)$ can be identified with the map $\colim V_\bullet \to \Ynd{\cA}(x)$, where $V_\bullet$ is the underlying simplicial object of the \v{C}ech nerve of~$p$.
        Hence, $f \colon U \to \Ynd{\cA}(x)$ can be identified with a map $\colim U_\bullet \to \Ynd{\cA}(x)$, where $U_\bullet$ is the underlying simplicial object of the \v{C}ech nerve of the induced map $\coprod_{i \in I} \Ynd{\cA}[u_i] \to \Ynd{\cA}(x)$.
        Thus, a presheaf $F$ is $\eta$-local if and only if the following map is an equivalence:
        \[
            \begin{aligned}
                F x
                \longto \Map{\PSh (\cA)} (U, F)
                \simeq \Map{\PSh (\cA)}\Big(\colim_{\Op{\Delta}} U_\bullet, \, F\Big)
                \simeq \lim_{\Op{\Delta}} \, \Map{\PSh (\cA)}(U_\bullet, F),
            \end{aligned}
        \]
        which proves the first statement.
        Now assume that $\cA$ has pullbacks.
        Observe that for any $ [n] \in \Op{\Delta}$, $U_n$ is an iterated pullback of components of the original morphism $\coprod_{i \in I} \Ynd{\cA}(u_i) \to \Ynd{\cA}(x)$.
        Since $\cA$ has pullbacks, there exists a morphism $\widetilde{f} \colon \coprod_{i \in I} u_i \to x$ such that $\Ynd{\cA} (\widetilde{f}) \cong f$.
        Thus, $U_\bullet \simeq \Ynd{\cA} \circ \widetilde{U}_\bullet$ where $\widetilde{U}_\bullet$ is the same diagram iterated pullback but using components of $\widetilde{f}$.
        Thus, a presheaf $F$ is $\eta$-local if and only if the following map is an equivalence:
        \[
            \begin{aligned}
                F x
                \longto \lim_{\Op{\Delta}} \, \Map{\PSh (\cA)}(U_\bullet, F)
                \simeq \lim_{\Op{\Delta}} \, \Map{\PSh (\cA)}(\Ynd{\cA} \circ \widetilde{U}_\bullet, F)
                \simeq \lim_{\Delta} \, F \circ \Op{\widetilde{U}_\bullet},
            \end{aligned}
        \]
        as we wanted to prove.
    \end{proof}

    Let $\cA$ be a small $\infty$-category which  admits pullbacks and a Grothendieck topology $\cT$.
    Let $\Sigma' = (\Op{\cA}, \bL)$ be a limit sketch where the set $\bL$ consists of, for each covering sieve generated by a family ${\{u_i \to x\}}_{i \in I}$, a cone over the underlying simplicial object $\widetilde{U}_\bullet$ of the \v{C}ech nerve of the induced map $\coprod_{i \in I} u_i \to x$, with cone point $x \in \cA$.
    A model of $\Sigma'$ is a presheaf $F \colon \Op{\cA} \to \Spaces$ such that
    \[
        Fx \longto \lim_{\Delta} \, F \circ \widetilde{U}_\bullet
    \]
    is an equivalence for every $f \colon U \to \Ynd{\cA} (x)$.
    Hence, as before, $\Mod (\Sigma') \simeq \Mod (\Sigma) \simeq \Sh(\cA, \cT)$.
\end{example}

To finish this section, we also give an example of a sketch with a nonempty set of cocones:

\begin{example}
    Let $f \colon 0 \to 1$ be the generating morphism of $\sDelta{1}$.
    Define a sketch $\Sigma = (\sDelta{1}, \bL, \bC)$ where $\bL$ contains only a cone of the empty diagram with cone point $1$, and $\bC$ contains the cocone corresponding to the commutative square
    \[
        \begin{tikzcd}[ampersand replacement=\&,cramped]
            0 \& 1 \\
            1 \& {1\rlap{.}}
            \arrow["f", from=1-1, to=1-2]
            \arrow["f"', from=1-1, to=2-1]
            \arrow[equals, from=1-2, to=2-2]
            \arrow[equals, from=2-1, to=2-2]
        \end{tikzcd}
    \]
    Let $\cC$ be an $\infty$-category  with a terminal object $1_\cC$ and pushouts.
    Thus, a model on $\cC$ is a functor $F \colon \cA \to \cC$ such that $F1 \cong 1_\cC$ and the square
    \[
        \begin{tikzcd}[ampersand replacement=\&,cramped]
            F0 \& {1_\cC} \\
            {1_\cC} \& {1_\cC}
            \arrow[from=1-1, to=1-2]
            \arrow[from=1-1, to=2-1]
            \arrow[equals, from=1-2, to=2-2]
            \arrow[equals, from=2-1, to=2-2]
        \end{tikzcd}
    \]
    is a pushout.
    If we take models on an $\infty$-topos $\cE$, then applying $F$ is equivalent to choosing an object $F0 \in \cE$ such that the terminal map $Fa \to 1_\cE$ is an epimorphism.
    As a special case, as observed in \autocite{Hoyois2019},
    the models on $\Spaces$ are nonempty spaces $X$ whose suspension is contractible, i.e., nonempty path-connected spaces whose fundamental group $\pi_1(X)$ is perfect, that is, its abelianization is zero.
    The $\infty$-category of models of this sketch $\Sigma$ on $\mathcal{S}$ is not presentable, since it does not have an initial object.
\end{example}

\section{Flat functors}%
\label{sec:flat-functors}

Let $\cA$ be a small $\infty$-category, and let $\Spaces$ denote the $\infty$-category of spaces.
For notational purposes, we consider both the covariant Yoneda embedding $\Ynd{\cA} \colon \cA\to \Fun{\Op{\cA}}{\Spaces}$ and the contravariant Yoneda embedding $\coYnd{\cA}\colon \Op{\cA}\to \Fun{\cA}{\Spaces}$.
For a presheaf $F \colon \Op{\cA} \to \Spaces$, we denote by $\Lan F$ the left Kan extension of $F$ along $\coYnd{\cA}$.
We denote by $\hRelSlice{\cA}{F}$ the relative slice $\RelSlice{\left(\Ynd{\cA}\right)}{F}$ at a presheaf $F \colon \Op{\cA} \to \Spaces$ along the Yoneda functor, i.e., the following pullback:
\[
    \begin{tikzcd}[ampersand replacement=\&,cramped]
        {\hRelSlice{\cA}{F}} \& {\Slice{\PSh (\cA)}{F}} \\
        \cA \& {\PSh (\cA)}\rlap{.}
        \arrow[from=1-1, to=1-2]
        \arrow[from=1-1, to=2-1]
        \arrow["\lrcorner"{anchor=center, pos=0.125}, draw=none, from=1-1, to=2-2]
        \arrow[from=1-2, to=2-2]
        \arrow["{\Ynd{\cA}}"', from=2-1, to=2-2]
    \end{tikzcd}
\]
Thus, the objects of $\hRelSlice{\cA}{F}$ are pairs consisting of an object $a \in \cA$ and a natural transformation $\alpha \colon \Ynd{\cA} (a) \Rightarrow F$, or, equivalently, pairs of an object $a \in \cA$ and an object $\alpha \in F a$.

Given a regular cardinal $\kappa$, a presheaf $F \colon \Op{\cA} \to \Spaces$ is \textit{$\kappa$-flat} if the left Kan extension
\[
    \Lan F \colon \Fun{\cA}{\Spaces} \longrightarrow \Spaces
\]
preserves $\kappa$-small limits.
The following theorem characterizes $\kappa$-flatness and generalizes results of~\autocite{Adamek2002,Adamek1994, Kelly1982,Makkai1989} to the context of $\infty$-categories.

\begin{theorem}%
    \label{thm:flat<->filtcolim}
    Let $\cA$ be a small $\infty$-category, $\kappa$ a regular cardinal, and $F \colon \Op{\cA} \to \Spaces$ a presheaf.
    The following statements are equivalent:
    \begin{enumerate}[label={\rm (\roman*)}]
        \item $F$ is $\kappa$-flat.
        \item $\Lan F$ preserves $\kappa$-small limits of representables.
        \item $\hRelSlice{\cA}{F}$ is a $\kappa$-filtered $\infty$-category.
        \item $F$ is a $\kappa$-filtered colimit of representables.
    \end{enumerate}
\end{theorem}
\begin{proof}
    First, ${\rm (i)} \Rightarrow {\rm (ii)}$ holds by definition.
    The implication ${\rm (iii)} \Rightarrow {\rm (iv)}$ follows from the fact that any presheaf $F$ is a canonical colimit of representables, with diagram
    \[
        \hRelSlice{\cA}{F} \overset{U}{\longto} \cA \xrightarrow{h_\cA} \Fun{\Op{\cA}}{\Spaces}.
    \]
    Next we prove that ${\rm (iv)} \Rightarrow {\rm (i)}$.
    If $a \in \cA$, then, by \autoref{thm:yoneda-lemma}, $\Ev_a$ is a left Kan extension of $\Ynd{\cA}(a)$ along~$\coYnd{\cA}$:
    \[
        \begin{tikzcd}[ampersand replacement=\&,cramped]
            {\Op{\cA}} \& \Spaces\rlap{.} \\
            {\Fun{\cA}{\Spaces}}
            \arrow["{\Ynd{\cA}(a)}", from=1-1, to=1-2]
            \arrow["{\coYnd{\cA}}"', hook, from=1-1, to=2-1]
            \arrow["{\Ev_a}"', from=2-1, to=1-2]
        \end{tikzcd}
    \]
    Assume that $F$ is a $\kappa$-filtered colimit of representables, i.e.,
    \[
        F \cong \colim_{i \in \cI} \, \Ynd{\cA} D(i),
    \]
    where $\cI$ is a $\kappa$-filtered $\infty$-category and $D \colon \cI \to \cA$.
    Consider its left Kan extension
    \[
        \begin{tikzcd}[ampersand replacement=\&,cramped]
            {\Op{\cA}} \& \Spaces\rlap{.} \\
            {\Fun{\cA}{\Spaces}}
            \arrow["F", from=1-1, to=1-2]
            \arrow["{\coYnd{\cA}}"', hook, from=1-1, to=2-1]
            \arrow["{\Lan F}"', from=2-1, to=1-2]
        \end{tikzcd}
    \]
    By definition, $\Lan$ is a left adjoint and preserves colimits.
    Hence, using that $F$ is a $\kappa$-filtered colimit of representables, we obtain that
    \[
        \Lan F
        \cong \Lan \left(\colim_{i \in \cI} \, \Ynd{\cA} D(i)\right)
        \cong \colim_{i \in \cI} \, \Lan (\Ynd{\cA} D(i))
        \cong \colim_{i \in \cI} \, \Ev_{D(i)}.
    \]
    By~\autocite[Proposition 5.1.2.3]{Lurie2009}, the evaluation functor $\Ev_a$ preserves colimits and limits.
    Since $\cI$ is $\kappa$\nobreakdash-filtered, colimits indexed by $\cI$ on $\Spaces$ commute with $\kappa$-small limits.
    Therefore, $\Lan F$ preserves $\kappa$-small limits.

    Finally, we prove that ${\rm (ii)} \Rightarrow \rm{(iii)}$.
    Assume that $\Lan F \colon \Fun{\cA}{\Spaces} \to \Spaces$ preserves $\kappa$-small limits of representables.
    To prove that $\hRelSlice{\cA}{F}$ is a $\kappa$-filtered $\infty$-category, consider any diagram $D \colon \cK \to \hRelSlice{\cA}{F}$ where $\cK$ is $\kappa$-small, and we aim to show that $D$ has a cocone.
    Since $\hRelSlice{\cA}{F}$ classifies~$F$, we have the following pullbacks:
    \[
        \begin{tikzcd}[ampersand replacement=\&,cramped]
            {\Op{\cK}} \& {\Op{(\hRelSlice{\cA}{F})}} \& {\Spaces_{/*}} \& {\Fun{\sDelta{1}}{\Spaces}} \\
            \& {\Op{\cA}} \& \Spaces \& {\Spaces\times\Spaces}\rlap{.}
            \arrow["{\Op{D}}", from=1-1, to=1-2]
            \arrow[from=1-2, to=1-3]
            \arrow["{\Op{U}}"', from=1-2, to=2-2]
            \arrow["\lrcorner"{anchor=center, pos=0.125}, draw=none, from=1-2, to=2-3]
            \arrow[from=1-3, to=1-4]
            \arrow[from=1-3, to=2-3]
            \arrow["\lrcorner"{anchor=center, pos=0.125}, draw=none, from=1-3, to=2-4]
            \arrow[from=1-4, to=2-4]
            \arrow["F"', from=2-2, to=2-3]
            \arrow["{*\times \Id}"', from=2-3, to=2-4]
        \end{tikzcd}
    \]
    The composition of the top row can be viewed as a natural transformation between functors $\Op{\cK} \to \cS$, which is a cone of the diagram $F \Op{U} \Op{D}$ with cone point $*$.
    Using the universal property of limits in $\Spaces$ applied to cones with cone point $*$, we obtain an equivalence
    \[
        \Map{\sFun{\Op{\cK}}{\Spaces}} (\const *, F \Op{U} \Op{D})
        \simeq \Map{\Spaces} \Big(*,\, \lim_{\Op{\cK}} \, F \Op{U} \Op{D}\Big),
    \]
    which determines an object $x \in \lim_{\Op{\cK}} F \Op{U} \Op{D}$.
    Here $\delta *$ denotes the constant functor at~$*$.

    Consider the diagram
    \[
        \coYnd{\cA} \Op{U} \Op{D} \colon \Op{\cK} \longto \Fun{\cA}{\Spaces},
    \]
    where $U \colon \hRelSlice{\cA}{F} \to \cA$ is the forgetful functor of $\hRelSlice{\cA}{F}$ and $\coYnd{\cA}$ is the contravariant Yoneda embedding.
    Since $\Fun{\cA}{\Spaces}$ is complete, there exists a limit $H = \lim_{\Op{\cK}} \coYnd{\cA} \Op{U} \Op{D} \in \Fun{\cA}{\Spaces}$.
    Then
    \[
        \Lan F (H)
        = \Lan F \Big(\lim_{\Op{\cK}} \, \coYnd{\cA} \Op{U} \Op{D}\Big)
        \simeq \lim_{\Op{\cK}} \, \Lan F (\coYnd{\cA} \Op{U} \Op{D})
        \simeq \lim_{\Op{\cK}} \, F \Op{U} \Op{D},
    \]
    where the first equivalence follows from $\Lan F$ preserving $\kappa$-small limits of representables, and the second one from the isomorphism $(\Lan{F}) \circ \coYnd{\cA} \cong F$ defining a left Kan extension. We denote by $\gamma\colon \Lan F (H)
        \to \lim_{\Op{\cK}} F \Op{U} \Op{D}$ the composite equivalence.

    In addition, the left Kan extension $\Lan F (H)$ can be computed as a colimit of
    \[
        \hRelSlice{\cA}{F} \overset{U}{\longto} \cA \overset{H}{\longto} \Spaces.
    \]
    Consequently, there is an object $\Inv{\gamma} (x)$ in $\colim_{\hRelSlice{\cA}{F}} H U$.

    Since $H U$ is a diagram in $\Spaces$, by~\autocite[Lemma 6.2.3.13]{Lurie2009}, the induced map
    \[
        \coprod_{(a, \alpha) \in \hRelSlice{\cA}{F}} \!\!\! H a
        \overset{\theta}{\longto} \colim_{\hRelSlice{\cA}{F}} H U,
    \]
    which is defined using the universal cocones $\theta_{(a, \alpha)}$ of the colimit for each $(a, \alpha)$, is an effective epimorphism.
    In addition, by \autocite[Corollary 7.2.1.15]{Lurie2009}, the induced map
    \[
        \coprod_{(a, \alpha) \in \hRelSlice{\cA}{F}} \!\!\! \pi_0 (H a)
        \longto \pi_0 \Big( \colim_{\hRelSlice{\cA}{F}} H U \Big)
    \]
    is surjective.
    Hence, there exist $(a, \alpha) \in \hRelSlice{\cA}{F}$ and $y \in H a$ such that $\theta_{(a, \alpha)} (y) \cong \Inv{\gamma} (x)$.

    By definition, any object $y \in H a$ can be viewed as a cone $\tilde{y} \colon \delta a \Rightarrow \Op{U}\Op{D}$ in~$\Op{\cA}$, using the equivalences
    \[\begin{aligned}
            \Map{\Spaces} (* , Ha)
             & \simeq \Map{\Spaces} \Big(*,\, \Big(\lim_{\Op{\cK}} \, \coYnd{\cA} \Op{U} \Op{D}\Big)(a)\Big)
            \simeq \Map{\Spaces} \Big(*,\, \lim_{\Op{\cK}} \, \Map{\Op{\cA}}(a, \Op{U}\Op{D}(-))\Big)        \\[0.1cm] &
            \simeq \Map{\sFun{\Op{\cK}}{\Spaces}} (\const *, \,\Map{\Op{\cA}}(a, \Op{U}\Op{D}(-)))
            \simeq \Map{\sFun{\Op{\cK}}{\Op{\cA}}} (\delta a, \, \Op{U}\Op{D}).
        \end{aligned}\]
    Thus, we have the following commutative diagram:
    \[
        \begin{tikzcd}[ampersand replacement=\&,cramped]
            {\Op{\cK}} \& {\Op{(\hRelSlice{\cA}{F})}} \& {\coSlice{\Spaces}{*}} \\
            {\Cones{{(\Op{\cK}})}} \& {\Op{\cA}} \& \Spaces\rlap{.}
            \arrow["{\Op{D}}", from=1-1, to=1-2]
            \arrow[hook, from=1-1, to=2-1]
            \arrow[from=1-2, to=1-3]
            \arrow["{\Op{U}}"', from=1-2, to=2-2]
            \arrow["\lrcorner"{anchor=center, pos=0.125}, draw=none, from=1-2, to=2-3]
            \arrow[from=1-3, to=2-3]
            \arrow["{\tilde{y}}"', from=2-1, to=2-2]
            \arrow["F"', from=2-2, to=2-3]
        \end{tikzcd}
    \]
    To construct a cone on $\Op{(\hRelSlice{\cA}{F})}$, we need two cones: one on $\Op{\cA}$ and one on $\Op{(\coSlice{\Spaces}{*})}$, compatible with the pullback square.
    To obtain a cone on $\Op{(\coSlice{\Spaces}{*})}$ with cone point $(F a, \alpha)$, it is enough to construct a map
    \[
        F a \longrightarrow \lim_{\Op{\cK}} \, F \Op{U} \Op{D}
    \]
    sending $\alpha$ to $x$, because, by \autocite[ Corollary 4.3.1.11]{Lurie2009}, the limit $\lim_{\Op{\cK}} F \Op{U} \Op{D}$ lifts to a limit on $\Op{(\coSlice{\Spaces}{*})}$ with cone point $(\lim_{\Op{\cK}} F \Op{U} \Op{D}, x)$.

    Let $\bar{y} \colon \coYnd{\cA} (a) \to H$ be a natural transformation corresponding to $y \in Ha$ by the Yoneda Lemma.
    By applying $\Lan F$ to $\bar{y}$ and composing with $\gamma$, we obtain a map
    \[
        \mu_y \colon Fa
        \simeq \Lan F (\coYnd{\cA} (a))
        \xrightarrow{\Lan F (\bar{y})} \Lan F (H)
        \overset{\gamma}{\longto} \lim_{\Op{\cK}} F \Op{U} \Op{D}.
    \]
    Using the expression of the left Kan extension $\Lan F$ as a colimit, we obtain a commutative square in $\Fun{\cA}{\Spaces}$:
    \[
        \begin{tikzcd}[ampersand replacement=\&,cramped]
            {\coYnd{\cA} (a)} \&\& H \\
            {\llap{$\const \Lan F (\coYnd{\cA} (a)) \cong$}\displaystyle \const \colim_{\hRelSlice{\cA}{F}}((\coYnd{\cA} (a)) U)} \&\& {\displaystyle \const \colim_{\hRelSlice{\cA}{F}}HU\rlap{$\cong \const \Lan F (H)$,}}
            \arrow["{\bar{y}}", from=1-1, to=1-3]
            \arrow["\bar{\alpha}", from=1-1, to=2-1]
            \arrow["{\bar\theta}", from=1-3, to=2-3]
            \arrow["{\const \Lan F (\bar{y})}", from=2-1, to=2-3]
        \end{tikzcd}
        \smallskip
    \]
    where $\bar{\theta}$ is the universal cocone of $\colim_{\hRelSlice{\cA}{F}}HU$, and $\bar{\alpha}$ is the natural transformation corresponding to $\alpha \in Fa$ by the Yoneda Lemma.
    The commutativity of this square implies that $\Lan F (\bar{y}) (\alpha) \cong \theta(y)$, and, therefore, $\mu_y (\alpha) \cong \gamma \theta (y) \cong x$.
    Thus, there is a cone on $\Op{(\hRelSlice{\cA}{F})}$, which corresponds to a cocone on $\hRelSlice{\cA}{F}$, as we wanted to achieve.
\end{proof}

In the previous proof, we used that, if an $\infty$-category $\cI$ is $\kappa$\nobreakdash-filtered, then colimits indexed by $\cI$ on $\Spaces$ commute with $\kappa$-small limits, as shown in \autocite[Proposition 5.3.3.3]{Lurie2009}.
The converse is a direct corollary of the characterization of flat functors:

\begin{corollary}
    An $\infty$-category $\cI$ is $\kappa$-filtered if and only if the colimits of shape $\cI$ on $\Spaces$ commute with all $\kappa$-small limits in~$\Spaces$, i.e., if and only if the functor
    $
        \colim_\cI \colon \Fun{\cI}{\Spaces} \longrightarrow \Spaces
    $
    preserves $\kappa$-small limits.
\end{corollary}
\begin{proof}
    Since $\const *$ is the terminal object of $\PSh (\cI)$, there are equivalences
    \[
        \Slice{\PSh (\cI)}{\const *}
        \simeq \PSh (\cI)
        \quad\text{ and }\quad
        \hRelSlice{\cI}{\const *}
        \simeq \cI,
    \]
    where the second equivalence is a pullback of the first one.
    Then, $\cI$ is $\kappa$-filtered if and only if $\hRelSlice{\cI}{\const *}$ is $\kappa$-filtered, and the colimit functor is the left Kan extension along the Yoneda embedding of the constant functor:
    \[
        \Lan (\const *)
        \cong \colim_{\hRelSlice{\cI}{\const *}} (- \circ U)
        \cong \colim_\cI (-).
    \]
    By~\autoref{thm:flat<->filtcolim}, the flatness of $\const *$, i.e., the fact that $\colim_{\cI}$ preserves $\kappa$-small limits, is equivalent to the $\kappa$-filteredness of $\hRelSlice{\cI}{\const *}$, and thus to the $\kappa$-filteredness of $\cI$.
\end{proof}

Under the assumptions of~\autoref{thm:flat<->filtcolim}, if in addition $\cA$ is $\kappa$-cocomplete, then the following simpler characterization of $\kappa$-flat functors follows as a corollary:

\begin{corollary}\label{lem:indcont}
    Let $\cA$ be a small $\infty$-category, $\kappa$ be a regular cardinal and $F \colon \Op{\cA} \to \Spaces$ be a presheaf on $\cA$.
    If $F$ is $\kappa$-flat, then it preserves all $\kappa$-small limits in $\Op{\cA}$.
    Conversely, if $\cA$ is $\kappa$-cocomplete and $F$ preserves all $\kappa$-small limits in $\Op{\cA}$, then $F$ is $\kappa$-flat.
\end{corollary}
\begin{proof}
    Assume that $F \colon \Op{\cA} \to \Spaces$ is $\kappa$-flat, and consider a diagram $D \colon \cK \to \Op{\cA}$ with limit in $\Op{\cA}$ and where $\cK$ is $\kappa$-small.
    Then, there are isomorphisms
    \[\begin{aligned}
            \lim_\cK F D
             & \cong \lim_\cK \, \Lan F (\coYnd{\cA} D)
            \cong \Lan F (\lim_\cK \coYnd{\cA} D)       \\[0.1cm]
             & \cong \Lan F (\coYnd{\cA} (\lim_\cK D))
            \cong F (\lim_\cK D),
        \end{aligned}\]
    where the first and the last follow from the isomorphism $\Lan F (\coYnd{\cA}) \cong F$, while the second follows from the fact that $\Lan F$ preserves $\kappa$-small limits, and the third is given by the fact that $\coYnd{\cA}$ preserves the limits that exist in~$\Op{\cA}$.

    Conversely, assume that $\cA$ is $\kappa$-cocomplete and $F$ preserves $\kappa$-small limits in $\Op{\cA}$.
    The right fibration associated to $F$ is the pullback $\Slice{(\Op{\Spaces})}{*}  \times_{\Op{\Spaces}} \cA$ of the universal right fibration by $\Op{F}$.
    Since $\Spaces$ is cocomplete, $\cA$ is $\kappa$-cocomplete and $F$ preserves $\kappa$-small limits, by~\autocite[Lemma 5.4.5.5]{Lurie2009}, $\Slice{(\Op{\Spaces})}{*}  \times_{\Op{\Spaces}} \cA$ is $\kappa$-complete.
    Thus, it is $\kappa$-filtered, and, by~\autoref{thm:flat<->filtcolim}, $F$ is $\kappa$-flat.
\end{proof}

Denote by $\FunL{\cA}{\cC}$ the full subcategory of $\Fun{\cA}{\cC}$ spanned by functors which are left adjoints.
If $\cA$ is locally small and $\cC$ is presentable, then the Adjoint Functor Theorem implies that the left adjoints are precisely the functors preserving small colimits.
Recall the universal property of the $\infty$-category of presheaves (see~\autocite[Theorem 5.1.5.6]{Lurie2009}): if $\cC$ is a cocomplete $\infty$-category, the composition with the Yoneda embedding $\Ynd{\cA} \colon \cA \to \PSh (\cA)$ induces an equivalence
\[
    \FunL{\PSh (\cA)}{\cC}
    \overset{\sim}{\longto}
    \Fun{\cA}{\cC}.
\]

Using Theorem~\ref{thm:flat<->filtcolim} together with the universal property of presheaves, we can prove new characterizations of accessibility, which generalize results  from~\autocite{Adamek2002,Adamek1994, Makkai1989}:

\begin{theorem}\label{lem:acc-flat}
    Let $\cC$ be an $\infty$-category.
    There exist some small $\infty$-category $\cA$ and some regular cardinal $\kappa$ such that the following are equivalent:
    \begin{enumerate}[label={\rm (\roman*)}]
        \item $\cC$ is accessible.
        \item $\cC$ is equivalent to the full subcategory of $\PSh (\cA)$ spanned by the $\kappa$-filtered colimits of representables.
        \item $\cC$ is equivalent to the full subcategory of all functors from $\Fun{\cA}{\Spaces}$ to $\Spaces$ preserving colimits and $\kappa$-filtered limits.
        \item $\cC$ is equivalent to the full subcategory of all functors from $\Fun{\cA}{\Spaces}$ to $\Spaces$ preserving colimits and $\kappa$-filtered limits of representables.
    \end{enumerate}
\end{theorem}
\begin{proof}
    The equivalence (i) $\Leftrightarrow$  (ii) is well known, and can be found in~\autocite[Corollary 5.3.5.4]{Lurie2009} and~\autocite[Subsection 11.7]{Rezk2022}.
    In both cases, the small $\infty$-category $\cA$ is equivalent to the $\infty$-category of $\kappa$-compact objects of $\cC$.

    Therefore, we need to prove that (ii) $\Leftrightarrow$ (iii) $\Leftrightarrow$ (iv).
    Let $\cA$ be a small $\infty$-category and $\kappa$ be a regular cardinal.
    Define $\cD$ as the full subcategory of $\PSh (\cA)$ spanned by the $\kappa$-filtered colimits of representables, and $\cE$ (resp.\ $\cF$) as the full subcategory of all functors from $\Fun{\cA}{\Spaces}$ to $\Spaces$ preserving colimits and $\kappa$-filtered limits (resp.\ $\kappa$-filtered limits of representables).
    We want to show equivalences between these three $\infty$-categories.

    In principle, $\cE$ and $\cF$ are full subcategories of the large $\infty$-category $\Fun{\Fun{\cA}{\Spaces}}{\Spaces}$, and therefore they could be large.
    Because $\Spaces$ is presentable and $\cA$ is locally small, the $\infty$-category of functors from $\Fun{\cA}{\Spaces}$ to $\Spaces$ preserving colimits is equivalent to $\FunL{\Fun{\cA}{\Spaces}}{\Spaces}$.
    Then, $\cE$ (resp.\ $\cF$) is equivalent to the full subcategory of $\FunL{\Fun{\cA}{\Spaces}}{\Spaces}$ spanned by those functors preserving $\kappa$-filtered limits (resp.\ $\kappa$-filtered limits of representables).
    By the universal property of the $\infty$-category of presheaves,
    \[
        \FunL{\Fun{\cA}{\Spaces}}{\Spaces}
        = \FunL{\PSh (\Op{\cA})}{\Spaces}
        \simeq \Fun{\Op{\cA}}{\Spaces}
        = \PSh (\cA).
    \]
    Thus, $\cE$ (resp.\ $\cF$) are also the full subcategories of $\PSh (\cA)$ spanned by those presheaves whose left Kan extension preserves $\kappa$-filtered limits (resp.\ $\kappa$-filtered limits of representables).
    In addition, $\cE$ and $\cF$ must also be locally small $\infty$-categories.

    Recall that two full subcategories of the same $\infty$-category are equivalent if they have isomorphic sets of objects.
    By~\autoref{thm:flat<->filtcolim}, a presheaf is a $\kappa$-filtered colimit of representables if and only if it is $\kappa$-flat, if and only if the left Kan extension preserves $\kappa$-filtered limits of representables.
    Then, because $\cD$, $\cE$ and $\cF$ are full subcategories of $\PSh (\cA)$ with isomorphic sets of objects, they must be equivalent.
\end{proof}

A natural question to ask is what this characterization looks like when considering presentable $\infty$-categories instead of accessible ones.
We use the notation $\Cont_\kappa(\Op{\cA})$ for the full subcategory of functors $\Op{\cA} \to \Spaces$ which preserve all limits that exist in $\Op{\cA}$.
The previous proposition amounts to an inclusion of full subcategories
\[
    \Ind_\kappa (\cA)
    \simeq \Flat_\kappa (\cA)
    \subseteq \Cont_\kappa(\Op{\cA}).
\]
Furthermore, if $\cA$ admits $\kappa$-small colimits, then the inclusion $\Flat_\kappa (\cA)
    \subseteq \Cont_\kappa(\Op{\cA})$ becomes an equivalence of full subcategories $\Flat_\kappa (\cA)
    \simeq \Cont_\kappa(\Op{\cA})$.

This equivalence yields the following corollary, which coincides with the characterization found in~\autocite[Proposition 5.3.5.4]{Lurie2009}:

\begin{corollary}\label{cor:presentable<->Cont}
    Let $\cC$ be an $\infty$-category and $\kappa$ be a regular cardinal.
    Then $\cC$ is $\kappa$-presentable if and only if it is equivalent to $\Cont_\kappa (\cA)$ for some small $\infty$-category $\cA$ which admits $\kappa$-small colimits.
\end{corollary}

\section{\texorpdfstring{Limit-sketchable $\infty$-categories}{Limit-sketchable infinity-categories}}%
\label{sec:limit-sketchable}

Our goal in this section is to generalize the well-known characterization of presentable categories as limit-sketchable categories to the higher setting.
Thus, we aim to prove that an $\infty$-category is presentable if and only if it is equivalent to the $\infty$-category of models of a limit sketch.

\begin{theorem}\label{thm:presentable->lsketch}
    Every $\kappa$-presentable $\infty$-category is normally limit $\kappa$-sketchable.
\end{theorem}
\begin{proof}
    Any $\kappa$-presentable $\infty$-category $\cC$ is of the form $\cC \simeq \Ind_\kappa (\cA)$ for some small $\infty$-category $\cA$ which admits $\kappa$-small colimits.
    Consider a limit sketch $\Sigma = (\Op{\cA}, \bL)$ where $\bL$ is the set of all limit cones of $\kappa$-small diagrams in $\Op{\cA}$.
    Observe that $\bL$ is well-defined as a set because $\cA$ is small, and that $\Sigma$ is $\kappa$-small because all the diagrams in $\bL$ are so.
    The $\infty$-category of models $\Mod (\Sigma)$ is, by definition, the $\infty$-category of functors preserving all limit cones of $\kappa$-small diagrams in $\Op{\cA}$, i.e., $\Cont_\kappa (\Op{\cA})$.
    Since $\cA$ admits $\kappa$-small colimits, \autoref{cor:presentable<->Cont} implies that $\Ind_\kappa (\cA)$ is equivalent to $\Cont_\kappa (\Op{\cA})$ as full subcategories of $\PSh (\cA)$.
    Therefore, $\cC \simeq \Ind_\kappa (\cA)$ is equivalent to $\Cont_\kappa (\Op{\cA}) \simeq \Mod(\Sigma)$.
\end{proof}

\begin{theorem}\label{thm:lsketch->presentable}
    Let $\kappa$ be an uncountable regular cardinal and $\Sigma = (\cA, \bL)$ be a limit $\kappa$-sketch.
    Then:
    \begin{enumerate}[label={\rm (\alph*)}]
        \item $\Mod (\Sigma)$ is presentable and an accessible reflective localization of $\Fun{\cA}{\Spaces}$.\label{itm:models-adjunction}
        \item $\Mod (\Sigma) \subseteq \Fun{\cA}{\Spaces}$ is stable under $\kappa$-filtered colimits.
        \item If $L : \Fun{\cA}{\Spaces} \rightleftarrows \Mod (\Sigma) : i$ denotes the adjunction of~\ref{itm:models-adjunction}, where $i$ is the inclusion, then $i \circ L$ preserves $\kappa$-filtered colimits.
        \item $\Mod (\Sigma)$ is $\kappa$-presentable and a $\kappa$-accessible reflective localization of $\Fun{\cA}{\Spaces}$.
    \end{enumerate}
\end{theorem}
\begin{proof}
    If we prove~\ref{itm:models-adjunction} and (b), then (c) and (d) follow from~\autocite[Corollary 5.5.7.3]{Lurie2009}.
    Let us start by proving~\ref{itm:models-adjunction}.
    By~\autocite[Proposition 5.5.4.15]{Lurie2009} and since $\Fun{\cA}{\Spaces}$ is presentable, if we find a set of morphisms $M$ such that $\Mod (\Sigma) = \Loc (\Fun{\cA}{\Spaces}, M)$, then we may infer that $\Mod (\Sigma)$ is presentable and an accessible reflective localization of $\Fun{\cA}{\Spaces}$.

    Let $\alpha \colon \cK^\triangleleft \to \cA$ be a cone of $\bL$, with cone point $x \in \cA$ and diagram $D \colon \cK \to \cA$.
    Consider the composition of the cocone $\Op{\alpha}$ with the Yoneda embedding:
    \[
        \coYnd{\cA} \circ \Op{\alpha} :
        \Op{(\cK^\triangleleft)} \simeq {(\Op{\cK})}^\triangleright \longrightarrow \Op{\cA} \longrightarrow \Fun{\cA}{\Spaces}.
    \]
    Since $\Fun{\cA}{\Spaces}$ is cocomplete, the diagram $\coYnd{\cA} \circ \Op{D}$ has a colimit, denoted by $\colim_{\Op{\cK}} \coYnd{\cA} \circ \Op{D}$.
    By the universal property of colimits, since $\coYnd{\cA} \circ \Op{\alpha}$ is also a cocone with diagram $\coYnd{\cA} \circ \Op{D}$, there exist a natural transformation of functors
    \[
        m_\alpha \colon
        \colim_{\Op{\cK}} \, \coYnd{\cA} \circ \Op{D}
        \longto
        \coYnd{\cA} (x).
    \]
    We pick the collection of morphisms $M = {\{m_\alpha\}}_{\alpha \in \bL}$.
    Observe that $M$ is a set of the same cardinality as~$\bL$.
    A functor $F \colon \cA \to \Spaces$ is $M$-local if, for every $m_\alpha \in M$, the induced map
    \[
        m_\alpha^* \colon
        \Map{} (
        \coYnd{\cA} (x),
        F
        )
        \longto
        \Map{} (
        \colim_{\Op{\cK}} \, \coYnd{\cA} \circ \Op{D}, \,
        F
        )
    \]
    is an equivalence.
    We want to prove that $\Mod (\Sigma)$ coincides with the full subcategory of $\Fun{\cA}{\Spaces}$ consisting of $M$-local objects.
    Given a functor $F \colon \cA \to \Spaces$, we need to show that $F$ sends cones of $\bL$ to limit cones in $\Spaces$ if and only if $F$ is $M$-local.

    Let $F \colon \cA \to \Spaces$ be a functor and $\alpha \colon \cK^\triangleleft \to \cA$ be a cone of $\cL$, with cone point $x \in \cA$ and diagram $D \colon \cK \to \cA$.
    Since $\Spaces$ is complete, the diagram $F \circ D \colon \cK \to \Spaces$ has a limit.
    By the universal property of limits applied to the cone $F \circ \alpha \colon \cK^\triangleleft \to \Spaces$, there exists a map
    \[
        t_\alpha \colon F x \longto \lim_\cK (F \circ D).
    \]
    Then, $t_\alpha$ is an equivalence if and only if $F \in \Mod (\Sigma)$.
    We can repeat the previous process with the functor $\Map{\sFun{\cA}{\Spaces}} (\coYnd{\cA}(-), F)$ instead of $F$ to obtain a map
    \[
        \widehat{t}_\alpha \colon \Map{\sFun{\cA}{\Spaces}} (\coYnd{\cA}(x), F) \longto \lim_\cK \,\Map{\sFun{\cA}{\Spaces}} (\coYnd{\cA}(D(-)), F).
    \]
    By the naturality of the Yoneda embedding, $\widehat{t}_\alpha$ is an equivalence if and only if $t_\alpha$ is one.

    Thus, to show that $F \in \Mod (\Sigma)$ if and only if $F$ is $M$-local, it is sufficient to check that, for every $\alpha \in \bL$, $\widehat{t}_\alpha$ is an equivalence if and only if $m_\alpha^*$ is one.
    Observe that the functor $\Map{} (-, F) $ can be expressed as
    \[
        \Map{\sFun{\cA}{\Spaces}} (-, F)
        \cong \Ynd{\sFun{\cA}{\Spaces}}(F)
        \cong \Ev_F \circ\, \coYnd{\sFun{\cA}{\Spaces}},
    \]
    and hence it sends colimits to limits, since $\Ev_F$ preserves limits~\autocite[Proposition 5.1.2.3]{Lurie2009} and $\coYnd{\sFun{\cA}{\Spaces}}$ sends colimits to limits~\autocite[Proposition 5.1.3.2]{Lurie2009}.
    Thus, $\Map{} (\displaystyle \colim (\coYnd{\cA} \circ D), F)$ is a limit of the diagram $\Map{} (\coYnd{\cA} \circ D, F)$.
    By the uniqueness of limits, there is an equivalence
    \[
        \sigma \colon \Map{} (\displaystyle \colim (\coYnd{\cA} \circ D), F) \longrightarrow \lim \Map{} (\coYnd{\cA} \circ D, F).
    \]
    Since the two objects are limits of $\Map{} (\coYnd{\cA} \circ D, F)$, the following diagram commutes:
    \[
        \begin{tikzcd}[ampersand replacement=\&,cramped]
            {\Map{} (\coYnd{\cA} (x), F)} \&\& {\Map{} (\displaystyle \colim (\coYnd{\cA} \circ D), F)} \\
            \&\& {\displaystyle \lim \, \Map{} (\coYnd{\cA} \circ D, F)\rlap{.}}
            \arrow["{m_\alpha^*}", from=1-1, to=1-3]
            \arrow["{\widehat{t}_\alpha}"'{pos=0.6}, from=1-1, to=2-3]
            \arrow["\sigma"', dashed, from=1-3, to=2-3]
            \arrow["\simeq", dashed, from=1-3, to=2-3]
        \end{tikzcd}
    \]
    Therefore, by the two-out-of-three property, $\widehat{t}_\alpha$ is an equivalence if and only if $m_\alpha^*$ is one, for every $\alpha \in \bL$.

    Thus, it only remains to prove (b), i.e., $\Mod (\Sigma) \subseteq \Fun{\cA}{\Spaces}$ is stable under $\kappa$-filtered colimits.
    Let $\cI$ be a $\kappa$-filtered $\infty$-category, and $F \colon \cI \to \Mod (\Sigma) \subseteq \Fun{\cA}{\Spaces}$ be a $\kappa$-filtered diagram of models of $\Sigma$.
    We want to see that the colimit of $F$ on $\Fun{\cA}{\Spaces}$ is also a model of $\Sigma$, which is equivalent to seeing that the maps
    \[
        m_\alpha^* \colon
        \Map{} (
        \coYnd{\cA} (x), \,
        \colim_\cI F
        )
        \longto
        \Map{} (
        \colim_{\Op{\cK}} \,\coYnd{\cA} \circ \Op{D},\,
        \colim_\cI F
        )
    \]
    are equivalences for every $\alpha \in \bL$, with cone point $x \in \cA$ and diagram $D \colon \cK \to \cA$.

    Recall that $\Fun{\cA}{\Spaces}$ is $\aleph_0$-presentable (see~\autocite[Proposition 5.3.5.12]{Lurie2009}) with the essential image of the Yoneda embedding being the $\aleph_0$-compact objects.
    In particular, since $\aleph_0 \leq \kappa$, every image of the Yoneda embedding is also a $\kappa$-compact object.
    By~\autocite[Corollary 5.3.4.15]{Lurie2009}, a $\kappa$-small colimit of $\kappa$-compact objects is $\kappa$-compact.
    Since $\Sigma$ is a limit $\kappa$-sketch, $\cK$ is $\kappa$-small.
    Thus, $\coYnd{\cA} (x)$ and $\colim_{\Op{\cK}} \coYnd{\cA} \circ \Op{D}$ are $\kappa$-compact objects.
    Therefore, using the commutativity of $\kappa$-filtered colimits with maps from $\kappa$-compact objects, the maps
    \[
        m_\alpha^* = \Map{} (m_\alpha, \, \colim_\cI F) \colon
        \Map{} (
        \coYnd{\cA} (x), \,
        \colim_\cI F
        )
        \longto
        \Map{} (
        \colim_{\Op{\cK}} \,\coYnd{\cA} \circ \Op{D}, \,
        \colim_\cI F
        )
    \]
    are equivalent to
    \[
        \colim_{i\in\cI}  \, \Map{} (m_\alpha, F(i)) \colon
        \colim_{i\in\cI} \,
        \Map{} (
        \coYnd{\cA} (x),
        F(i)
        )
        \longto
        \colim_{i\in\cI} \,
        \Map{} (
        \colim_{\Op{\cK}} \coYnd{\cA} \circ \Op{D}, \,
        F(i)
        ).
    \]
    Since each $F(i)$ is a model of $\Sigma$, the maps $\Map{} (m_\alpha, F(i))$ are equivalences for every $i \in \cI$ and $\alpha \in \bL$.
    Thus, the maps $m_\alpha^* = \colim_{i\in\cI} \, \Map{} (m_\alpha, F(i))$ are equivalences, as we wanted to prove.
\end{proof}

Combining the two previous theorems, we recover the characterization of presentable $\infty$\nobreakdash-cat\-egories as the limit-sketchable ones, and a normalization theorem for limit sketches:

\begin{corollary}\label{cor:main01}
    An $\infty$-category $\cC$ is $\kappa$-presentable, where $\kappa$ is a regular cardinal, if and only if $\cC$ is limit $\kappa$-sketchable.
\end{corollary}

\begin{corollary}
    For every limit sketch $\Sigma$, the $\infty$-category $\Mod (\Sigma)$ is complete and cocomplete.
\end{corollary}

\begin{corollary}
    For every limit $\kappa$-sketch $\Sigma$, there exists a normal limit $\kappa$-sketch $\Theta$ such that
    $
        \Mod (\Sigma) \simeq \Mod (\Theta).
    $
\end{corollary}
\begin{proof}
    By~\autoref{thm:lsketch->presentable}, we know that $\Mod (\Sigma)$ is presentable.
    Thus, \autoref{thm:presentable->lsketch} yields a normal limit sketch $\Theta = (\Op{\Mod^\kappa (\Sigma)}, \bL)$ such that $\Mod (\Sigma) \simeq \Mod (\Theta)$, where $\Mod^\kappa (\Sigma)$ are the $\kappa$-compact objects in $\Mod (\Sigma)$ and $\bL$ is the set of all $\kappa$-small limits of $\Mod^\kappa (\Sigma)$.
\end{proof}

\begin{example}
    Every $\infty$-topos is a left-exact accessible reflective localization of $\PSh (\cA)$ for some small $\infty$-category~$\cA$.
    Therefore, \autoref{thm:presentable->lsketch} implies that every $\infty$-topos is limit-sketchable.

    The $\infty$-category $\Sh (\cA, \cT)$ of sheaves on a small $\infty$-category $\cA$ equipped with a Grothendieck topology $\cT$ is a special case.
    A~more explicit sketch whose $\infty$-category of models is equivalent to $\Sh (\cA, \cT)$ has been given in Example~\ref{ex:sheaves}.
\end{example}

Let $\cC$ and $\cD$ be two presentable $\infty$-categories.
According to~\autocite[Section 4.8.1]{Lurie2017}, the $\infty$-category of presentable categories has a symmetric monoidal structure given by the Lurie tensor product, which is defined as
\[
    \cC \otimes \cD
    = \Op{\FunL{\cC}{\Op{\cD}}},
\]
and has $\Spaces$ as unit.
Given two limit sketches $\Sigma = (\cA, \bA)$ and $T = (\cB, \bB)$, we can define the following limit sketch~(see~\autocite[Exercise 1.l.2]{Adamek1994} and~\autocite[Notation 4.8.1.7]{Lurie2017}):
\[
    \Sigma \boxtimes T = (
    \cA \times \cB,\;
    \bA \boxtimes \bB
    )
    \quad\text{and}\quad
    \bA \boxtimes \bB = (\bA \times \Obj (\cB))
    \sqcup (\Obj(\cA) \times \bB).
\]
Let $\kappa$ and $\lambda$ be two regular cardinals, and let $\mu = \max (\kappa, \lambda)$.
Observe that, given a limit $\kappa$-sketch $\Sigma$ and a limit $\lambda$-sketch $T$, all cones in the sketch $\Sigma \boxtimes T$ have $\mu$-small diagrams.
Hence, $\Sigma \boxtimes T$ is a limit $\mu$-sketch.

Let $\mathbb{1} = (\sDelta{0}, \emptyset)$ denote the trivial sketch, which has models $\Mod (\mathbb{1}) = \Fun{\sDelta{0}}{\Spaces} = \Spaces$.
Then, taking models of a limit sketch built with the $\boxtimes$ construction is compatible with the tensor product of presentable $\infty$-categories in the following way:

\begin{proposition}\label{prop:odot}
    For limit sketches $\Sigma = (\cA, \bA)$ and $T = (\cB, \bB)$, there are equivalences
    \[\begin{gathered}
            \Mod (\Sigma \boxtimes T)
            \simeq \Mod(\Sigma,\Mod (T))
            \simeq \Mod (\Sigma) \otimes \Mod (T)\rlap{, \text{ and}}\\
            \Mod (\Sigma \boxtimes \mathbb{1})
            \simeq \Mod (\Sigma)
            \simeq \Mod (\Sigma) \otimes \Spaces.
        \end{gathered}\]
\end{proposition}
\begin{proof}
    The second chain of equivalences follows from the first one by taking $T$ to be $\mathbb{1}$.
    Let $\FunInv{\cA \times \cB}{\Spaces}{\bA \, \boxtimes\, \bB}$ denote the full subcategory $\Fun{\cA \times \cB}{\Spaces}$ spanned by those functors sending the cones of $\bA$ to limits in the first variable, and the cones of $\bB$ to limits in the second.
    First observe that
    \[
        \Mod (\Sigma, \Mod (T))
        = \FunInv{\cA}{\FunInv{\cB}{\Spaces}{\bB}}{\bA}
        \simeq \FunInv{\cA \times \cB}{\Spaces}{\bA\, \boxtimes\, \bB}
        = \Mod (\Sigma \boxtimes T),
    \]
    where the second equivalence follows since $\iCat$ is cartesian closed.
    Recall that by~\autocite[Proposition 5.2.6.2]{Lurie2009}, there is an equivalence $\Op{\FunL{\cC}{\Op{\cD}}} \simeq \FunR{\Op{\cD}}{\cC}$ for any $\infty$-categories $\cC$ and $\cD$.
    Then, using the previous equivalence together with the unit of the tensor and the fact that $\iCat$ is cartesian closed, the proposition follows from the following chain of equivalences:
    \[\begin{aligned}
            \Mod (\Sigma) \otimes \Mod (T)
             & = \Op{\FunL{\Mod (\Sigma)}{\Op{\Mod (T)}}}
            \simeq \FunInv{\Op{\Mod (T)}}{\Mod (\Sigma)}{R}                        \\
             & = \FunInv{\Op{\Mod (T)}}{\FunInv{\cA}{\Spaces}{\bA}}{R}
            \simeq \FunInv{\Op{\Mod (T)} \times \cA}{\Spaces}{R \,\boxtimes\, \bA} \\
             & \simeq \FunInv{\cA}{\FunInv{\Op{\Mod (T)}}{\Spaces}{R}}{\bA}
            \simeq \FunInv{\cA}{\Op{\FunInv{\Spaces}{\Op{\Mod (T)}}{L}}}{\bA}      \\
             & = \FunInv{\cA}{\Spaces \otimes \Mod (T)}{\bA}
            \simeq \FunInv{\cA}{\Mod (T)}{\bA} = \Mod (\Sigma, \Mod (T)),
        \end{aligned}\]
    where $\FunInv{\Op{\Mod (T)} \times \cA}{\Spaces}{R\, \boxtimes\, \bA}$ denotes the full subcategory $\Fun{\Op{\Mod (T)} \times \cA}{\Spaces}$ spanned by those functors which are right adjoints in the first variable, and send the cones of $\bA$ to limits in the second.
\end{proof}

The previous result can be used to generalize~\autoref{thm:lsketch->presentable} to models over any presentable $\infty$-category:

\begin{corollary}
    Let $\kappa$ and $\lambda$ be two regular cardinals, and let $\mu = \max (\kappa, \lambda)$.
    If $\cC$ is a $\kappa$\nobreakdash-pres\-entable $\infty$-category, then, for every limit $\lambda$-sketch $\Sigma = (\cA, \bL)$, the $\infty$-category of models $\Mod (\Sigma, \cC)$ over $\cC$ is $\mu$-presentable, and it is equivalent to $\Mod (\Sigma) \otimes \cC$.
\end{corollary}
\begin{proof}
    By~\autoref{thm:presentable->lsketch}, since $\cC$ is $\kappa$-presentable, it is equivalent to $\Mod (T)$ for a normal limit $\kappa$-sketch~$T$.
    Then it follows from Proposition~\ref{prop:odot} that
    \[
        \Mod (\Sigma, \cC) \simeq \Mod (\Sigma \boxtimes T) \simeq \Mod (\Sigma) \otimes \cC,
    \]
    where $\Sigma \boxtimes T$ is a limit $\mu$-sketch by definition.
    Hence, using~\autoref{thm:lsketch->presentable}, we conclude that $\Mod (\Sigma, \cC)$ is $\mu$-presentable.
\end{proof}

From the examples given in Section~\ref{sec:sketches},
we can conclude, using Theorem~\ref{thm:lsketch->presentable},
that the following full subcategories of any presentable $\infty$-category $\cC$ are presentable:
pointed objects, spectrum objects, pre-category objects, univalent category objects, monoid objects, groupoid objects, group objects, commutative monoid objects, and abelian group objects.

Consequently, the following $\infty$-categories are presentable: pointed $\infty$-groupoids, spectra, Segal spaces, complete Segal spaces, $A_\infty$-spaces,
grouplike $A_\infty$-spaces, $A_\infty$-ring spectra,  $E_\infty$-spaces, infinite loop spaces, $E_\infty$-ring spectra, and higher sheaves over any Grothendieck topology.

Although these categories are extensively discussed in various forms throughout the literature, their sketchability is seldom explicitly addressed.
The examples in Section~\ref{sec:sketches} of this article draw on similar examples from unpublished work of Joyal~\autocite{JoyalChicago,JoyalCRM}.
Our treatment of complete Segal spaces, dendroidal Segal spaces, and complete dendroidal Segal spaces and higher sheaves in Examples~\ref{ex:comsegal},~\ref{ex:dendroidal},~\ref{ex:oper-object} and~\ref{ex:sheaves} is~new.

\section{\texorpdfstring{Sketchable $\infty$-categories}{Sketchable infinity-categories}}%
\label{sec:sketchable}

In this section, we describe the relation between accessibility and sketchability.
Specifically, we prove that $\Mod (\Sigma, \cC)$ is accessible when $\cC$ is presentable, and that every accessible $\infty$-category can be modeled by a normal sketch.

\begin{theorem}\label{thm:accessible->sketch}
    Every $\kappa$-accessible $\infty$-category is normally $\kappa$-sketchable.
\end{theorem}
\begin{proof}
    Every $\kappa$-accessible $\infty$-category $\cC$ satisfies $\cC \simeq \Ind_\kappa (\cA)$ for some small $\infty$-category~$\cA$.
    Denote by $\widehat{\cA} \subseteq \Fun{\cA}{\Spaces}$ the free $\kappa$-small cocompletion of $\Op{\cA}$, in the sense of~\autocite{Rezk2022a}.
    Observe that the Yoneda embedding factors through $\widehat{\cA}$ as
    \[
        \coYnd{\cA}
        \colon \Op{\cA}
        \overset{j}{\longto} \widehat{\cA}
        \overset{i}{\longcof} \Fun{\cA}{\Spaces}.
    \]
    The left Kan extension of the Yoneda embedding along itself is equivalent to the identity \autocite[Lemma~5.1.5.3]{Lurie2009}. This yields a canonical colimit for any presheaf $f \in \Fun{\cA}{\Spaces}$:
    \[
        \colim \left(
        \hRelSlice{\cA}{f}
        \overset{\pi}{\longto} \Op{\cA}
        \overset{j}{\longto} \widehat{\cA}
        \overset{i}{\longcof} \Fun{\cA}{\Spaces}
        \right)
        \cong \Lan \coYnd{\cA} (f)
        \cong f,
    \]
    where $\pi \colon \hRelSlice{\cA}{f} \to \Op{\cA}$ is the forgetful functor.
    Then, the restriction of $\Lan \coYnd{\cA}$ to $\widehat{\cA}$ factors through $\widehat{\cA}$ itself as the identity.
    Define a normal sketch $\Sigma = (\widehat{\cA}, \bL, \bC)$ where $\bL$ is the set of limits of $\kappa$-small diagrams in $\Op{\cA}$, and $\bC$ is the set of canonical colimits of the cone points in $\widehat{\cA}$ of all limit cones in~$\bL$.

    We want to show that $\Mod (\Sigma)$ is equivalent to $\cC$.
    The characterization of~\autoref{lem:acc-flat} shows that $\cC$ is equivalent to the $\infty$-category of functors from $\Fun{\cA}{\Spaces}$ to $\Spaces$ preserving colimits and $\kappa$-filtered limits of representables, denoted by $\cE$.
    Hence, it suffices to show that $\Mod (\Sigma)$ is equivalent to $\cE$.
    Consider the restriction of functors of $\cE$ to $\widehat{\cA}$, i.e., the precomposition functor
    \[
        - \circ i
        \colon \cE \longto \Mod (\Sigma).
    \]
    By the universal property of the completion and $\kappa$-small completion of $\Op{\cA}$, we have that
    \[
        \FunL{\Fun{\cA}{\Spaces}}{\Spaces}
        \simeq \Fun{\Op{\cA}}{\Spaces}
        \simeq \FunL{\widehat{\cA}}{\Spaces}
        \cof \FunInv{\widehat{\cA}}{\Spaces}{\bC},
    \]
    where $\FunInv{\widehat{\cA}}{\Spaces}{\bC}$ denotes the $\infty$-category of functors preserving the colimits in $\bC$, and thus the last inclusion forgets about the preservation of other colimits.
    Recall that $\cE$ is the restriction of $\FunL{\Fun{\cA}{\Spaces}}{\Spaces}$ to functors which preserve $\kappa$-small limits of representables. Therefore, the restriction of $\FunL{\widehat{\cA}}{\Spaces}$ to functors preserving $\kappa$-small limits of representables is equal to the image of $- \circ i$, thanks to $\widehat{\cA}$ being the $\kappa$-small completion of $\Op{\cA}$.
    Observe that $\Mod (\Sigma)$ is the full subcategory of $\FunInv{\widehat{\cA}}{\Spaces}{\bC}$ spanned by the functors that preserve $\kappa$-small limits of representables.
    Hence, the image of $\cE$ by $- \circ i$ is a full subcategory of $\Mod (\Sigma)$ given by forgetting the preservation of other colimits, and therefore it is fully faithful.

    Consequently, we only need to show that $- \circ i$ is essentially surjective.
    For each model $F \in \Mod (\Sigma)$, we want to find a functor $G \in \cE$ such that $G \circ i \cong F$.
    Take as candidate the left Kan extension of the restriction $\restrict{F}{\Op{\cA}}$, denoted by
    \[
        G
        = \Lan \big(\restrict{F}{\Op{\cA}}\big)
        = \Lan (F \circ j) \colon \Fun{\cA}{\Spaces} \longrightarrow \Spaces,
    \]
    which preserves colimits by definition.
    Observe that, for every limit cone point $f \in \widehat{\cA}$ of a $\kappa$-small diagram in $\Op{\cA}$, we have
    \[
        G(f)
        = \Lan (F \circ j) (f)
        \cong \colim_{\hRelSlice{\cA}{f}} (F \circ j \circ \pi)
        \cong F (\colim_{\hRelSlice{\cA}{f}} \, j \circ \pi)
        \cong F (f),
    \]
    where the second isomorphism follows from the fact that $F$ preserves the canonical colimits of~$\bC$.
    Since $F$ preserves $\kappa$-small limits of $\Op{\cA}$ by the definition of~$\bL$, and $G$ coincides with $F$ on such presheaves, $G$ also preserves $\kappa$-small limits of $\Op{\cA}$, and therefore it belongs to $\cE$.
    Now, since $G \circ i$ preserves $\kappa$-small limits of $\Op{\cA}$, the universal property of a $\kappa$-small cocompletion of $\Op{\cA}$ implies that $G \circ i \cong F$, as we wanted to prove.
\end{proof}

\begin{theorem}\label{thm:sketch->accessible}
    If $\cC$ is a presentable $\infty$-category, then, for every sketch $\Sigma$, the $\infty$-category of models $\Mod (\Sigma, \cC)$ over $\cC$ is accessible.
\end{theorem}
\begin{proof}
    Let $\Sigma = (\cA, \bL, \bC)$ be a sketch.
    Observe that, by definition,
    \[
        \Mod (\Sigma, \cC)
        = \Mod (\Sigma_\bL, \cC) \cap
        \left(\,\textstyle\bigcap_{c \in \bC} \Mod (\Sigma_c, \cC)\right)
        \subseteq \Fun{\cA}{\cC},
    \]
    where $\Sigma_\bL$ is the limit sketch $(\cA, \bL)$, and $\Sigma_c$ is the colimit sketch $(\cA, \{c\})$ with only one cocone $c \in \bC$. These should be understood as intersections between full subcategories of $\Fun{\cA}{\cC}$.

    By~\autocite[Proposition 5.4.7.10]{Lurie2009}, every intersection of accessible reflective localizations of an accessible category is accessible.
    Since $\cC$ is presentable and $\cA$ is small, $\Fun{\cA}{\cC}$ is also presentable.
    Thus, proving that $\Mod (\Sigma, \cC)$ is accessible is equivalent to showing that $\Mod (\Sigma_\bL, \cC)$ and $\Mod (\Sigma_c, \cC)$ are accessible, and accessible reflective localizations of $\Fun{\cA}{\Spaces}$.
    By~\autoref{thm:lsketch->presentable}, $\Mod (\Sigma_\bL, \cC)$ is presentable and an accessible reflective localization of $\Fun{\cA}{\Spaces}$, because $\Sigma_\bL$ is a limit sketch and $\cC$ is presentable.

    Given a cocone $c \in \bC$ with $c \colon \coCones{D} \to \cA$, there is a functor $\widetilde{c} \colon \Fun{\cA}{\cC} \to \Fun{\sDelta{1}}{\cC}$ defined by the composition
    \[
        \Fun{\cA}{\cC}
        \overset{c^*}{\longto}
        \Fun{\coCones{D}}{\cC}
        \overset{L}{\longto}
        \Fun{\coCones{(\coCones{D})}}{\cC}
        \overset{i^*}{\longto}
        \Fun{\sDelta{1}}{\cC},
    \]
    where $i \colon \sDelta{0} \join \sDelta{0} \to D \join \sDelta{0} \join \sDelta{0}$ is the canonical inclusion, which induces a morphism
    \[
        i^*\colon
        \Fun{\coCones{(\coCones{D})}}{\cC}
        \simeq
        \Fun{D \join \sDelta{0} \join \sDelta{0}}{\cC}
        \longto
        \Fun{\sDelta{0} \join \sDelta{0}}{\cC}
        \simeq
        \Fun{\sDelta{1}}{\cC},
    \]
    and $L$ is the colimit-forming functor, left adjoint to the forgetful functor from cocones of diagram $\coCones{D}$ on $\cC$, which must exist because $\cC$ is cocomplete.
    Furthermore, since $c^*$, $L$ and $i^*$ are left adjoints, $\widetilde{c}$ is accessible.

    Let $J$ be the nerve of the two-point connected groupoid and $j \colon \sDelta{1} \to J$ be the groupoid completion of $\sDelta{1}$.
    By~\autocite[Corollary~2.1.1.8]{Land2021}, $\Fun{J}{\cC}$ corresponds to the full subcategory of $\Fun{\sDelta{1}}{\cC}$ spanned by the isomorphisms of~$\cC$.
    Because $\cC$ is presentable and $\sDelta{1}$ is small, $j^* \colon \Fun{J}{\cC} \to \Fun{\sDelta{1}}{\cC}$ is a left adjoint.
    Then, $\Mod (\Sigma_c, \cC)$ can be seen as a pullback
    \[
        \begin{tikzcd}[ampersand replacement=\&,cramped]
            {\Mod (\Sigma_c, \cC)} \& {\Fun{J}{\cC}} \\
            {\Fun{\cA}{\cC}} \& {\Fun{\sDelta{1}}{\cC}}\rlap{.}
            \arrow[from=1-1, to=1-2]
            \arrow["i"', hook, from=1-1, to=2-1]
            \arrow["\lrcorner"{anchor=center, pos=0.125}, draw=none, from=1-1, to=2-2]
            \arrow["{j^*}", hook, from=1-2, to=2-2]
            \arrow["{\widetilde{c}}", from=2-1, to=2-2]
        \end{tikzcd}
    \]
    This is a pullback of accessible $\infty$-categories and accessible functors, since left adjoints between accessible $\infty$-categories are accessible~\autocite[Proposition~5.4.7.7]{Lurie2009}.
    Thus, by~\autocite[Proposition~5.4.6.6]{Lurie2009}, $\Mod (\Sigma_c, \cC)$ is also accessible and the inclusion $i \colon \Mod (\Sigma_c, \cC) \to \Fun{\cA}{\cC}$ is an accessible functor.
    Therefore, $\Mod (\Sigma_c, \cC)$ is an accessible reflective localization.
\end{proof}

The previous theorem does \emph{not} prove that the $\infty$-category of models of a $\kappa$-sketch is $\kappa$\nobreakdash-acces\-sible.
This does not follow from our proof,  because~\autocite[Proposition 5.4.6.6]{Lurie2009} only proves that the pullback of a $\kappa$-accessible diagram is $\mu$-accessible for some $\mu \geq \kappa$.
In fact, there are $1$-categorical examples (\autocite[Remark 2.59]{Adamek1994} and \autocite[Example 3.3.6]{Makkai1989}) of $\aleph_0$-sketches with a category of models which is not $\aleph_0$-accessible.

\begin{corollary}\label{cor:main02}
    An $\infty$-category is accessible if and only if it is (normally) sketchable.
\end{corollary}

\begin{corollary}
    For every sketch $\Sigma$ there exists a normal sketch $\Theta$ such that $\Mod(\Sigma) \simeq \Mod(\Theta)$.
\end{corollary}
\begin{proof}
    By~\autoref{thm:sketch->accessible}, we know that $\Mod (\Sigma)$ is $\kappa$-accessible for some regular cardinal $\kappa$.
    Thus, we infer from \autoref{thm:accessible->sketch} that there exists a normal $\kappa$-sketch $\Theta = (\Op{\cA}, \bL, \bC)$ such that $\Mod (\Sigma) \simeq \Mod (\Theta)$, where $\cA$ is the free $\kappa$-small cocompletion of the $\infty$-category of $\kappa$-compact objects in $\Mod (\Sigma)$, and $\bL$ is the set of all $\kappa$-small limits of diagrams in $\Mod^\kappa (\Sigma)$.
\end{proof}

The $\infty$-category of models of a sketch $\Sigma = (\cA, \bL, \bC)$ can be viewed as the intersection of the full subcategories of $\Fun{\cA}{\cC}$ spanned by the models of the limit sketch $\Sigma_\bL = (\cA, \bL)$ and those of the colimit sketch $\Sigma_\bC = (\cA, \bC)$, respectively:
\[
    \Mod (\Sigma, \cC)
    = \Mod (\Sigma_\bL, \cC)
    \cap \Mod (\Sigma_\bC, \cC).
\]
If $\cC$ is presentable, then, by~\autoref{thm:lsketch->presentable}, $\Mod (\Sigma_\bL, \cC)$ is presentable, and by~\autoref{thm:sketch->accessible}, $\Mod (\Sigma_\bC, \cC)$ is accessible.
\autoref{prop:modiscomplete} implies that $\Mod (\Sigma_\bC, \cC)$ is cocomplete, and therefore it is also presentable.
Thus, $\Mod (\Sigma, \cC)$ is the intersection of two presentable $\infty$-categories.

\printbibliography

\vspace{1cm}

\noindent
Departament de Matemàtiques i Informàtica, Universitat de Barcelona, Gran Via de les Corts Catalanes 585, 08007 Barcelona, Spain

\medskip

\noindent
carles.casacuberta@ub.edu, javier.gutierrez@ub.edu, dvmcarpena@pm.me

\end{document}

%% file: commands.tex
\newcommand*{\longto}{\longrightarrow}
\newcommand*{\cof}{\hookrightarrow}
\newcommand*{\longcof}{\lhook\joinrel\longrightarrow}
\newcommand*{\colim}{\operatorname*{colim}}
\newcommand*{\restrict}[2]{{%
\left.\kern-\nulldelimiterspace%
#1%
\vphantom{\big|}%
\right|_{#2}%
}}

\newcommand*{\CatStyle}[1]{{\bf #1}}

\newcommand*{\N}{\mathbb{N}}

\newcommand*{\sSet}{\CatStyle{sSet}}

\newcommand*{\Spaces}{\CatStyle{\mathcal{S}}}
\newcommand*{\iCat}{\CatStyle{Cat_\infty}}
\newcommand*{\Spectra}{\CatStyle{Sp}}

\newcommand*{\sDelta}[1]{\Delta^{#1}}
\newcommand*{\dFace}[2]{\delta_{#2}}
\newcommand*{\dDege}[2]{\sigma_{#2}}
\newcommand*{\sFace}[2]{d_{#2}}

\newcommand*{\Obj}{\operatorname{Ob}}

\newcommand*{\Op}[1]{{#1}^{\operatorname{op}}}
\newcommand*{\Inv}[1]{{#1}^{-1}}
\newcommand*{\Id}{\operatorname{id}}
\newcommand*{\Hom}[1]{\operatorname{Hom}_{#1}}
\newcommand*{\Ev}{\operatorname{ev}}
\newcommand*{\pr}{\operatorname{pr}}
\newcommand*{\Mod}{\operatorname{Mod}}
\newcommand*{\PSh}{\operatorname{PSh}}
\newcommand*{\Sh}{\operatorname{Sh}}
\newcommand*{\Cont}{\operatorname{Cont}}
\newcommand*{\Ind}{\operatorname{Ind}}
\newcommand*{\Flat}{\operatorname{Flat}}
\newcommand*{\Loc}{\operatorname{Loc}}
\newcommand*{\Lan}{\operatorname{Lan}}
\newcommand*{\const}{\operatorname{\delta}}

\newcommand*{\ho}{\operatorname{h}}
\newcommand*{\join}{\operatorname{\star}}
\newcommand*{\Cones}[1]{{#1}^{\vartriangleleft}}
\newcommand*{\coCones}[1]{{#1}^{\vartriangleright}}
\newcommand*{\Map}[1]{\operatorname{Map}_{#1}}
\newcommand*{\gFun}[4]{\operatorname{Fun}^{#3}_{#4} (#1, #2)}
\newcommand*{\Fun}[2]{\gFun{#1}{#2}{}{}}
\newcommand*{\sFun}[2]{\Fun{#1}{\, {#2}}}
\newcommand*{\FunInv}[3]{\gFun{#1}{#2}{#3}{}}
\newcommand*{\FunL}[2]{\gFun{#1}{#2}{L}{}}
\newcommand*{\FunR}[2]{\gFun{#1}{#2}{R}{}}
\newcommand*{\Ynd}[1]{h_{#1}}
\newcommand*{\coYnd}[1]{h^{#1}}
\newcommand*{\Slice}[2]{{#1}_{/#2}}
\newcommand*{\coSlice}[2]{{#1}_{#2/}}
\newcommand*{\RelSlice}[2]{\Slice{#1}{#2}}
\newcommand*{\hRelSlice}[2]{\Slice{#1}{#2}}
\newcommand*{\cRelSlice}[2]{\Slice{#1}{#2}}

\newcommand*{\cA}{\mathcal{A}}
\newcommand*{\cB}{\mathcal{B}}
\newcommand*{\cC}{\mathcal{C}}
\newcommand*{\cD}{\mathcal{D}}
\newcommand*{\cE}{\mathcal{E}}
\newcommand*{\cF}{\mathcal{F}}

\newcommand*{\cH}{\mathcal{H}}
\newcommand*{\cI}{\mathcal{I}}
\newcommand*{\cK}{\mathcal{K}}
\newcommand*{\cL}{\mathcal{L}}
\newcommand*{\cS}{\mathcal{S}}
\newcommand*{\cX}{\mathcal{X}}

\newcommand*{\cT}{\mathcal{T}}
\newcommand*{\bA}{\mathbb{A}}
\newcommand*{\bB}{\mathbb{B}}
\newcommand*{\bL}{\mathbb{L}}
\newcommand*{\bC}{\mathbb{C}}
\newcommand*{\bD}{\mathbb{D}}
\newcommand*{\Deltabf}{\text{\boldmath$\Delta$}}
\newcommand*{\Gammabf}{\text{\boldmath$\Gamma$}}
\newcommand*{\Omegabf}{\text{\boldmath$\Omega$}}